\documentclass[11pt, reqno]{amsart}
\usepackage{amsfonts}
\usepackage{latexsym}
\usepackage{amssymb}
\usepackage{amsmath}
\usepackage{color}
\usepackage{bbm}
\usepackage{tikz}
\usepackage{enumerate}
\usepackage{enumitem}
\usepackage{mathrsfs}
\usepackage{todonotes}
\usepackage{relsize}
\usepackage{xfrac}
\usepackage[scr=boondox]{mathalpha}
\usepackage{verbatim}
\usepackage[left=2.5cm, top=2.5cm,bottom=2.5cm,right=2.5cm]{geometry}

\newcommand{\R}{\mathbb R}
\newcommand{\N}{\mathbb N}

\newcommand{\E}{\mathbb E}
\newcommand{\Pro}{\mathbb P}

\newcommand{\vol}{\mathrm{vol}}

\def\dint{\textup{d}}
\newcommand{\SSS}{\ensuremath{{\mathbb S}}}

\newcommand{\B}{\ensuremath{{\mathbb B}}}
\newcommand{\D}{\ensuremath{{\mathbb D}}}

\newtheorem{thm}{Theorem}[section]
\newtheorem{cor}[thm]{Corollary}
\newtheorem{lemma}[thm]{Lemma}
\newtheorem{df}[thm]{Definition}
\newtheorem{proposition}[thm]{Proposition}

{
\theoremstyle{definition}

\newtheorem{rmk}[thm]{Remark}

}
\def\bC{\mathbf{C}}

\def\bN{\mathbf{N}}
\def\bP{\mathbf{P}}

\def\bU{\mathbf{U}}

 	\definecolor{Kgreen}{rgb}{0.0, 0.5, 0.0}

\usepackage{nomencl}
\makenomenclature
\begin{document}

%%%%%%%%%%%%%%%%%%%%%%%%%%%%%%%%%%%%%%%%%%%%%5

\title[Sharp Asymptotics for $q$-Norms of Random Vectors in High-Dimensional $\ell_p^n$-Balls ]{Sharp Asymptotics for $q$-Norms of Random Vectors\\ in High-Dimensional $\ell_p^n$-Balls}

\author[Tom Kaufmann]{Tom Kaufmann}
\address{Tom Kaufmann: Faculty of Mathematics, Ruhr University Bochum, Germany} \email{tom.kaufmann@rub.de}


%\thanks{}

%\date{\today}

\begin{abstract}
Sharp large deviation results of Bahadur--Ranga Rao type are provided for the $q$-norm of random vectors distributed on the $\ell _{p}^{n}$-ball ${\mathbb{B}}^{n}_{p}$ according to the cone probability measure or the uniform distribution for $1 \le q<p < \infty $, thereby furthering previous large deviation results by Kabluchko, Prochno and Th\"{a}le in the same setting. These results are then applied to deduce sharp asymptotics for intersection volumes of different $\ell _{p}^{n}$-balls in the spirit of Schechtman and Schmuckenschl\"{a}ger, and for the length of the projection of an $\ell _{p}^{n}$-ball onto a line with uniform random direction. The sharp large deviation results are proven by providing convenient probabilistic representations of the $q$-norms, employing local limit theorems to approximate their densities, and then using geometric results for asymptotic expansions of Laplace integrals to integrate these densities and derive concrete probability estimates.
\end{abstract}
\maketitle
%
%
%\text{}\\
%\\
%\\
%\\
%\tableofcontents
%
%\newpage
%
%\vspace{0.5cm}
\section{Introduction}\label{sec:Introduction}
The study of convex bodies in high dimensions, known today as asymptotic
geometric analysis, has arisen from the local theory of Banach spaces,
which aimed at analyzing infinite-dimensional normed spaces via their finite-dimensional
substructures, such as their unit balls. Despite having its origin in the
realm of functional analysis, the field has since established itself in
its own right, considering problems also beyond the study of centrally
symmetric convex bodies that occur naturally as the unit balls of Banach
spaces. In high dimensions convex bodies exhibit certain regularities,
such as concentration of measure phenomena (see, e.g.,
\cite{GuedonConcentrationPhenomena}), which make it highly useful to approach
them from a probabilistic perspective. As pointed out in
\cite{BookAGA}, it might seem counter-intuitive to analyze something exhibiting
regularities from a probabilistic perspective, as probability concerns
itself with studying the nature of irregularity, i.e. randomness, of given
quantities. But as with well-known limit theorems from probability such
as the law of large numbers and the central limit theorem, with large sample
sizes (and analogously -- with high dimensionality) random objects exhibit
interesting patterns well characterized in the language of probability
and vice versa. Many results analogous to those from classic probability
have been found for high-dimensional convex sets, such as the central limit
theorem (see, e.g., Anttila, Ball and Perissinaki \cite{ABPclt}, Klartag
\cite{KlartagCLT,KlartagCLT2}). For further background on high-dimensional
convexity, see
\cite{BookAGA,BookGICB,GuedonConcentrationPhenomena,GuedonConcentrationIneq}.%

The $\ell _{p}^{n}$-ball ${\mathbb{B}}^{n}_{p}$,
$n \in \mathbb{N}$, has been a prominent object of study, as it is the unit
ball of the (finite-dimensional) sequence space $\ell _{p}^{n}$, and has
been the subject of a multitude of results. We will name only a select
few and refer to the survey by Prochno, Th\"{a}le and Turchi
\cite{PTTSurvey} for a comprehensive summary of classic and contemporary
results. Let us denote by $\mathbf{U}_{n,p}$ the uniform distribution
on the Euclidean $\ell _{p}^{n}$-ball
${\mathbb{B}}_{p}^{n}$ and by $\mathbf{C}_{n,p}$ the cone
probability measure on the $\ell _{p}^{n}$-sphere
${\mathbb{S}}_{p}^{n-1}$. Schechtman and Zinn
\cite{SchechtmanZinn} and Rachev and R\"{u}schendorf
\cite{RachevRueschendorf} showed a generalization of the Poincar\'{e}--Maxwell--Borel
lemma, proving that, for $k \in \mathbb{N}$ with $k < n$, the $k$-dimensional
marginal distribution of a random vector distributed according to
$\mathbf{C}_{n,p}$ converges to a $k$-dimensional generalized Gaussian
distribution as $n$ increases. They also provided a probabilistic representation for such
random vectors in terms of these generalized Gaussian distributions, which
will be a key building block in our main results. The primary quantity
of interest of this paper however is the behaviour of the $q$-norm
$\|Z\|_{q}$ of a random vector $Z$ in
${\mathbb{S}}_{p}^{n-1}$ and
${\mathbb{B}}_{p}^{n}$. This was first studied by Schechtman
and Zinn \cite{SchechtmanZinn}, who derived concentration inequalities
for $\|Z\|_{q}$ with $Z \sim \mathbf{C}_{n,p}$ and
$Z \sim \mathbf{U}_{n,p}$ for $q>p$. This is closely related to the intersection
volume of $t$-multiples of volume-normalized $\ell _{p}^{n}$-balls
${\mathbb{D}}_{p}^{n}:= \mathrm{vol}_{n}(
{\mathbb{B}}_{p}^{n})^{-1/n} {\mathbb{B}}_{p}^{n}$,
i.e.
$\text{vol}_{n}({\mathbb{D}}_{p}^{n} \cap t
{\mathbb{D}}_{q}^{n})$ with $t\in [0, \infty )$, for which
Schechtman and Schmuckenschl\"{a}ger
\cite{SchechtmanSchmuckenschlaeger} gave the asymptotics for
$t\ne 1$. Schechtman and Zinn \cite{SchechtmanZinn2000} expanded
%on
their
previous results in \cite{SchechtmanZinn}, by not only considering the
$q$-norm, but also images of random vectors under Lipschitz functions in
general. Thus, they gave concentration inequalities for $f(Z)$, with
$Z \sim \mathbf{C}_{n,p}$ and $Z \sim \mathbf{U}_{n,p}$,
$p\in [1,2)$, and $f$ a Lipschitz function with respect to the Euclidean
norm. Schmuckenschl\"{a}ger \cite{SchmuckCLT} provided a central limit
theorem (CLT) for $\|Z\|_{q}$ with $Z \sim \mathbf{C}_{n,p}$ and
$Z \sim \mathbf{U}_{n,p}$ and used it to refine the previous intersection
results in \cite{SchechtmanSchmuckenschlaeger} for all
$t\in (0,\infty )$. Naor \cite{NaorTAMS} gave concentration inequalities
for $\|Z\|^{q}_{q}$ with $Z \sim \mathbf{C}_{n,p}$, showed that the total
variation distance between $\mathbf{C}_{n,p}$ and the normalized surface
measure $\sigma _{n,p}$ on ${\mathbb{S}}_{p}^{n-1}$ tends
to zero proportional to $n^{-1/2}$, and used the previously mentioned results
to show a concentration inequality for $\|Z\|^{q}_{q}$ with
$Z \sim \sigma _{n,p}$. He also discussed how
%similar concentration results
concentration results similar
to Schechtman and Zinn \cite{SchechtmanZinn2000} for $\|Z\|_{q}$ could
already be derived from previous results of Gromov and Milman
\cite{Gromov} for the concentration of Lipschitz functions on convex bodies.
Kabluchko, Prochno and Th\"{a}le \cite{KPTLimitThm} gave a multivariate
CLT for $(\|Z\|_{q_{1}}, \ldots , \|Z\|_{q_{d}})$ with
$Z \sim \mathbf{U}_{n,p}$ in the spirit of \cite{SchmuckCLT} and also
considered the asymptotics for the intersection volume of multiple
$\ell _{p}^{n}$-balls, i.e.
$\text{vol}_{n}({\mathbb{D}}_{p}^{n} \cap t_{1}
{\mathbb{D}}_{q_{1}}^{n} \cap \cdots \cap t_{d}
{\mathbb{D}}_{q_{d}}^{n})$ with
$t_{i} \in [0, \infty )$. This CLT was furthermore applied by the same
authors to infer a central limit theorem for the length of
${\mathbb{B}}_{p}^{n}$ projected onto a line with uniform
random direction. Moreover, they provided a large deviation principle (LDP)
for $\|Z\|_{q}$ with $Z \sim \mathbf{C}_{n,p}$ and
$Z \sim \mathbf{U}_{n,p}$. In a follow-up paper \cite{KPTLimitThm2}, the
same authors showed a CLT for $\|Z\|_{q}$, where the distribution of
$Z$ is taken from a wider class of $p$-radial distributions
$\mathbf{P}_{n,p, \textbf{W}}$, introduced by Barthe, Gu\'{e}don, Mendelson
and Naor \cite{BartheGuedonEtAl}, consisting of mixtures of
$\mathbf{U}_{n,p}$ and $\mathbf{C}_{n,p}$, combined via a measure
$\textbf{W}$ on $[0, \infty )$. This class contains both
$\mathbf{U}_{n,p}$ and $\mathbf{C}_{n,p}$, but also distributions corresponding
with geometrically interesting projections (see, e.g.,
\cite[Introduction, (iii)]{KPTLimitThm2}). Finally, they gave a moderate
and a large deviation principle for $\|Z\|_{q}$ with
$Z \sim \mathbf{P}_{n,p, \textbf{W}}$.

Generally, studying large deviations within asymptotic geometric analysis
has started fairly recently with Gantert, Kim and Ramanan \cite{GKR}, who
gave an LDP for projections of random points in $\ell _{p}^{n}$-balls with
distributions $\mathbf{C}_{n,p}$ and $\mathbf{U}_{n,p}$ onto both random
and fixed one-dimensional subspaces. Today, large deviations theory has
become a well-established toolbox in high-dimensional convex geometry,
giving rise to a plethora of results (see, e.g.,
\cite{APTldp,APTclt,KPTLimitThm,KPTSanov,KPTLimitThm2,KimPhD,KimRamanan}).
Recently, a new tool from large deviations theory was introduced to asymptotic
geometric analysis by Liao and Ramanan \cite{LiaoRamanan}. They gave sharp
large deviation (SLD) results in the spirit of Bahadur and Ranga Rao
\cite{Bahadur} and Petrov \cite{Petrov} for the projections of random points
in $\ell _{p}^{n}$-balls with distributions $\mathbf{C}_{n,p}$ and
$\mathbf{U}_{n,p}$ onto a fixed one-dimensional subspace. Other works
in asymptotic geometric analysis have also employed methods from sharp
large deviations theory as well, such as Kabluchko and Prochno
\cite{KP}, who derived asymptotic volumes for generalizations of
$\ell _{p}^{n}$-balls, known as Orlicz balls, and showed a Schechtman and
Schmuckenschl\"{a}ger-type result by considering intersection volumes of
Orlicz balls. Their results on Orlicz balls were then expanded upon by
Alonso-Guiterr\'{e}z and Prochno in \cite{AP}, who gave the exact asymptotic
volume of Orlicz balls and provided thin-shell concentrations for them,
augmenting their results into sharp asymptotics under certain conditions.%
While LDPs only give tail asymptotics on a logarithmic scale, the sharp
asymptotics provided by sharp large deviations theory can give tail estimates
for concrete values of $n \in \mathbb{N}$, which makes them significantly
more useful for practical applications. Moreover, a lot of idiosyncrasies
of the underlying distributions, that are drowned out on the LDP scale,
are still visible on the SLD scale, thus giving a deeper understanding
of the geometric interpretation of the quantities involved. This paper
will follow closely in the footsteps of Liao and Ramanan
\cite{LiaoRamanan} and establish SLD results for the $q$-norms of random
vectors with distribution $\mathbf{C}_{n,p}$ and
$\mathbf{U}_{n,p}$. Furthermore, we will use these results to expand
on
works of Schechtman and Schmuckenschl\"{a}ger
\cite{SchechtmanSchmuckenschlaeger}, Schmuckenschl\"{a}ger
\cite{SchmuckCLT}, and Kabluchko, Prochno and Th\"{a}le
\cite{KPTLimitThm} for intersection volumes of $\ell _{p}^{n}$-balls by
giving sharp asymptotics for
$\text{vol}_{n}({\mathbb{D}}_{p}^{n} \cap t
{\mathbb{D}}_{q}^{n})$ at a considerably improved rate for
$1 \le q<p < \infty $ and $t > C(p,q)$ bigger than some constant dependent
on $p$ and $q$ only. Additionally, we will also apply our results for
$\ell _{p}^{n}$-spheres to retain sharp asymptotics for the length of the
projection of an $\ell _{p}^{n}$-ball onto the line spanned by a uniform
random direction.\looseness=1

The paper will proceed as follows: in Section~\ref{sec:Preliminaries} some basic notation and
definitions will be provided while also giving some appropriate background
on the involved large deviations theory. Furthermore, we will recapitulate
some existing results that are relevant to this paper. In Section~\ref{sec:MainResults} we will present our main results regarding the
$q$-norms of random vectors on $\ell _{p}^{n}$-spheres and
$\ell _{p}^{n}$-balls. Also, we will present and prove their application
to intersections and one-dimensional projections of $\ell _{p}^{n}$-balls,
and outline the idea of the two central proofs. In Section~\ref{sec:ProbRep} we will reformulate the target probabilities from the
main results in terms of useful probabilistic representations, using well-established
representations of random vectors in $\ell _{p}^{n}$-balls of Schechtman
and Zinn \cite{SchechtmanZinn} and Rachev and R\"{u}schendorf
\cite{RachevRueschendorf}. In Section~\ref{sec:JointDensityEstimate} local
density approximations of these probabilistic representations will be provided.
In Sections~\ref{sec:ProofMainResult1} and~\ref{sec:ProofMainResult2} we
will
%then
prove the SLD results for $\ell _{p}^{n}$-spheres and
$\ell _{p}^{n}$-balls, respectively, by integrating over the density estimates.
For that, we will utilize some geometric results for asymptotic expansions
of Laplace integrals from Adriani and Baldi \cite{AdrianiBaldi} and Breitung
and Hohenbichler \cite{Breitung1989}.
\section{Preliminaries}\label{sec:Preliminaries}

\subsection{Notation and important distributions}\label{subsec:Distributions}
We denote by $\vol_d$ the $d$-dimensional Lebesgue measure on $\R^d$ and write $\mathcal{B}(\R^d)$ for the $\sigma$-field of Borel sets in $\R^d$. For a set $A \in \mathcal{B}(\R^d)$ we write $A^\circ, \overline{A},$ $\partial A$, and $A^c$ for the interior, closure, boundary and complement of $A$, respectively. Furthermore, we write $ \langle \, \cdot \, , \, \cdot \,  \rangle$ for the standard scalar product in $\R^d$. %\\
For $g: \R^d \to \R^d$, we denote by $J_x g(x^*)$ the Jacobian of $g$ with respect to the vector $x$ evaluated at $x^* \in \R^d$, and for $f:\R^d \to \R$ by $\nabla_x f(x^*)$ and $\mathcal{H}_x f(x^*)$ the gradient and Hessian of $f$ with respect to the vector $x$ evaluated at $x^* \in \R^d$, respectively, and use the shorthand notation 
\begin{equation} \label{eq:AblNotation}
f_{[i_1, \ldots, i_d]}(x^*) = \displaystyle \frac{\partial^{i_1}}{\partial x_1^{i_1}} \ldots \frac{\partial^{i_d}}{\partial x_1^{i_d}} \, f (x)\big{|}_{x=x^*}.
\end{equation}
We write $(x_1, \ldots, x_d) \in \R^d$ for a standard column vector and for $x, y\in \R^d$, we write their product $x^Ty$ as $xy$, skipping the explicit transpose for brevity of notation. Given a random variable $X$ with distribution $\bP$, we write  $X \sim \bP$ and denote by $\E X$ its expectation. For two random variables $X,Y$ with the same distribution we write $X \overset{d}{=}Y$. For a random vector $X$ in $\R^d$ and $s \in \R^d$,  denote by $\varphi_X(s):= \E[e^{\langle s, X \rangle}]$ and $\Lambda_X(s) := \log \varphi_X(s)$ the moment generating function and cumulant generating function (m.g.f. and c.g.f.), respectively. We call the set of $s \in \R^d$ for which $\Lambda_X(s) < \infty$  the effective domain $\mathcal{D}_X$ of $\Lambda_X$.  Moreover, for $x \in \R^d$ we denote by $\Lambda_X^*(x) := \sup_{s \in \R^d} [\langle x, s \rangle - \Lambda_X(s)]$ the Legendre-Fenchel transform of the c.g.f. $\Lambda_X$. When considering sequences in $n \in \N$, we denote by $o(1)$ a sequence that tends to zero as $n \to \infty$. \\
\\
Let us consider the class of distributions at the core of the probabilistic constructions throughout this paper. We say a real-valued random variable $X$ has a generalized Gaussian distribution if its distribution has Lebesgue density 
$$\displaystyle f_{\textup{gen}}(x):=  \displaystyle \frac{b}{2 \, a\, \Gamma\left(\frac{1}{b}\right)} \, e^{-\big({|x- m|}/{a}\big)^b},\qquad x\in\R,$$
where $m \in \R$ and $a,b>0$, and denote this by $X \sim {\bN}_{\textup{gen}}(m, a, b)$. As mentioned in the introduction, the generalized Gaussian distributions are essential for constructing probabilistically equivalent representations of the quantities of interest, based on results of Schechtman and Zinn \cite{SchechtmanZinn} and Rachev and R\"uschendorf \cite{RachevRueschendorf}.  For these constructions we will be using the specific generalized Gaussian distribution $\bN_p := {\bN}_{\textup{gen}} \left(0, p^{1/p}, p \right)$ , $p \in [1, \infty),$ with density
$$
\displaystyle f_p(x) := \frac{1}{2 \, p^{1/p} \, \Gamma\big(1+\frac{1}{p}\big)}\, e^{-{|x|^p}/{p}},  \qquad x\in\R.
$$
For $X \sim \bN_p$ and $r >0$, we write $M_p(r):= \E |X|^r$ for the $r$-th absolute moment of $X$, for which it holds that 
\begin{equation} \label{eq:MomentXp}
M_p(r):= \E |X|^r = \left( \frac{p^{r/p}}{r+1} \, \frac{\Gamma \left(1 + \frac{r+1}{p}\right)}{\Gamma\left( 1+ \frac{1}{p}\right)} \right).
\end{equation}
%
%-------------------------------------------------------------------------------------------------------------------------------------------------------------------------------------------------
%
\subsection{Background material from (sharp) large deviations theory}\label{subsec:PrelimLDP} We will give some basic notions and definitions from large deviations theory. To keep this paper self-contained, we will present them here, while referring the reader to \cite{DZ,dH,Kallenberg} for additional background material on large deviations. Furthermore, we want to give some insight into the methods of the lesser known theory of sharp large deviations. 
\begin{df}
	Let $(\textup{\textbf{P}}_n)_{n\in\N}$ be a sequence of probability measures on $\R^d$. We say that $(\textup{\textbf{P}}_n)_{n\in\N}$ satisfies a large deviation principle (LDP) if there are two functions $s:\N\to\R$ and $\mathcal{I}:\R^d\to[0,\infty]$, such that $\mathcal{I}$ is lower semi-continuous and
	\begin{center}
		$
		\begin{array}{lllll}
		a)&\displaystyle \underset{n\to\infty}{\textup{lim sup}} \,\, \frac{1}{s(n)}\log \textup{\textbf{P}}_n(C) &\le& -\mathcal{I}(C) &\,\,\,\, \,\,\,\, \,\,\,\, \,\,\,\, \text{ for all } \,\, C\subset \R^n \text{ closed,}\\
		b)&\displaystyle \underset{n\to\infty}{\textup{lim inf}} \,\, \frac{1}{s(n)}\log \textup{\textbf{P}}_n(O) &\ge& -\mathcal{I}(O) &\,\,\,\, \,\,\,\, \,\,\,\, \,\,\,\, \text{ for all } \,\, O\subset \R^n \text{ open,}\\
		\end{array}
		$
	\end{center}
	where for $B \subset \R^d$ we define $\mathcal{I}(B) :=\inf_{x\in B}\mathcal{I}(x)$. We call $s$ the speed and $\mathcal{I}$ the rate function. We say that $\mathcal{I}$ is a good rate function, if it has compact sub-level sets.
\end{df}
We apply the definition of LDPs to sequences of random variables as well by applying the above definition to the sequence of their distributions. In our setting the sequence parameter $n \in \mathbb{N}$ will furthermore coincide with the space dimension $d \in \mathbb{N}$, as we are considering the effects of increasing dimensionality. Given a sequence $(X^{(n)})_{n \in \mathbb{N}}$ of i.i.d. random vectors in $\mathbb{R}^{d}$, one is frequently interested in the behaviour of the sequence $(S^{(n)})_{n \in \mathbb{N}}$ of the empirical averages $S^{(n)} := \frac{1}{n} \sum _{i=1}^{n} X^{(i)} \in \mathbb{R}^{d}$. One of the most well known and most frequently used results in the theory of large deviations is the theorem of Cram\'{e}r, which states that if the c.g.f. $\Lambda _{X}$ is finite in an open neighbourhood of the origin, then $(S^{(n)})_{n \in \mathbb{N}}$ satisfies an LDP in $\mathbb{R}^{d}$ with speed $n$ and rate function $\Lambda ^{*}_{X}$ (see, e.g., \cite[Theorem 2.2.30, Theorem 6.1.3, Corollary 6.1.6]{DZ}). Hence, under suitable exponential moment assumptions for the $X^{(n)}$, we can already infer the large deviation behaviour of $(S^{(n)})_{n \in \mathbb{N}}$.\\

The classic LDP gives us an idea of the asymptotic deviation behaviour of a sequence of distributions on a logarithmic scale. By doing this however, a lot of subtleties of the underlying distributions can be drowned out. Many small and medium scale properties of a given sequence of distributions are often missed in the asymptotic analysis of LDPs, since they either disappear for very large $n \in \mathbb{N}$ or are drowned out by other, more significant phenomena of the distribution. Thus, one is also interested in considering large deviations on a nonlogarithmic scale, which we refer to as ``sharp'' large deviations (also called ``precise'' or ``strong'' large deviations in the literature). One of the first and most prominent results in this regard was shown by Bahadur and Ranga Rao \cite{Bahadur}. They showed that for a sequence $(X^{(n)})_{n \in \mathbb{N}}$ of i.i.d. random variables and any $z > \mathbb{E}[X^{(n)}]$ with $\Lambda _{X}^{*}(z) < \infty $, it holds that
$$\Pro\left (S^{(n)} > z \right ) = \frac{1}{\sqrt{2 \pi n} \,  \kappa(z)  \xi(z)} e^{-n \Lambda_X^*(z)}  (1 + \textit{o}(1)),$$
where $\kappa(z)$ and $\xi(z)$ are only dependent on the distribution of the $X^{(n)}$ and the deviation size $z$. % 
%
% 
% ----------------------------------------------------------------
%
%
 This is proven via a (somewhat implicit) application of the the so-called saddle point method (or method of steepest descents), which was established by Debye \cite{Debye1909}, and brought to the realm of probability by Esscher \cite{Esscher1932} and Daniels \cite{Daniels1954}. The saddle point method generalizes Laplace's method for integral approximation to the complex plane, and is therefore highly useful when dealing with integrals over characteristic functions. In general, for appropriate functions $f,g$ and $n\in \N$ large, the saddle point method gives a way to approximate Laplace-type integrals $ \int_P g(z)e^{-nf(z)} \dint z$  along complex paths $P$, by deforming the integration path using Cauchy's theorem, into some $\tilde P$ that passes through a saddle point of $f$. The mass of the reformulated integral is then heavily concentrated around the saddle point and standard integral expansion methods, such as Edgeworth expansion, can be used to great effect. In the realm of probability, this has been used for both tail probabilities (e.g. Esscher \cite{Esscher1932}, Cram\'er \cite{Cramer1938}) and densities of random variables (e.g. Daniels \cite{Daniels1954}, Richter \cite{Richter1957, Richter1958}), by writing them as an integral over their characteristic functions, using the Fourier inversion formula, and then approximating those integrals via the use of a complex saddle point.  We say that this was used ``somewhat implicitly'' in certain results, such as those of Esscher \cite{Esscher1932}, Cram\'er \cite{Cramer1938} and Bahadur and Ranga Rao \cite{Bahadur}, since the technique used therein, which is a certain change of measure, often called exponential tilting or Esscher/Cram\'er transform, under the surface employs saddle points as well. For further background on this method, we refer to the book of Jensen \cite{JensenBook}.\\
 \\
 As mentioned in the introduction, Section \ref{sec:JointDensityEstimate} will provide density estimates for our probabilistic representations from Section \ref{sec:ProbRep}, which are derived using the saddle point method. However, since our probabilistic representations are given as sums of i.i.d. \hspace{-0.1cm}random vectors, we will refer to previous results where this was done explicitly, while making sure that the conditions for their application are still met in our setting. Generally, the core idea of the saddle point method, which is reformula-ting an integral such that all of its mass heavily concentrates around a critical point, around which we can then employ approximation methods, is used in the overall proof of our main results in a broader sense as well. We reformulate our target probabilities via some convenient representations, whose densities we also provide, such that the remaining integrals then heavily concentrate their mass around a given critical point, such that approximations at that point yield accurate results, as we will see in Sections \ref{sec:ProofMainResult1} and \ref{sec:ProofMainResult2}. 
\subsection{Distributions on $\ell_p^n$-balls}\label{subsec:GeometryLp}
For $p \in [1, \infty]$, $n\in\N$, and $x=(x_1,\ldots,x_n)\in\R^n$ we denote by
\begin{equation} \label{eq:LpNorm}
\|x\|_p := \begin{cases}
\Big(\sum\limits_{i=1}^n|x_i|^p\Big)^{1/p} &: p < \infty\\
\max\{|x_1|,\ldots,|x_n|\} &: p=\infty
\end{cases}
\end{equation}
the $\ell_p^n$-norm of $x$. Let $\B_p^n:=\{x\in\R^n:\|x\|_p\leq 1\}$ be the unit $\ell_p^n$-ball and $\SSS_p^{n-1}:=\{x\in\R^n:\|x\|_p=1\}$ be the unit $\ell_p^n$-sphere. We define the uniform distribution on $\B^n_p$ and cone probability measure on $\SSS^{n-1}_p$ as %
$$\bU_{n,p}(\,\cdot\,) := {\vol_n(\,\cdot\,)\over \vol_n(\B_p^n)}\qquad\text{and}\qquad\bC_{n,p}(\,\cdot\,) := {\vol_n(\{rx:r\in[0,1],x\in\,\cdot\,\})\over\vol_n(\B_p^n)}.$$
The following result is the basis of our probabilistic representations for random vectors with distribution $\bC_{n,p}$ and $\bU_{n,p}$ and is due to \cite{RachevRueschendorf} and \cite{SchechtmanZinn}. %
\begin{lemma}\label{lem:ProbRep}
Let $p \in [1, \infty)$, $Y=(Y_1,\ldots,Y_n)$ be a random vector in $\R^n$ with $Y_i \sim \bN_p$ i.i.d. \hspace{-0.2cm}, and $U$ an independent random variable uniformly distributed on $[0,1]$. Then,
	\begin{itemize}
	\item[i)]{ the random vector $  {Y}/{\|Y\|_p}$ has distribution $\bC_{n,p}$ and is independent of $\|Y\|_p,$}
	\vspace{0.25cm}
	\item[ii)]{ the random vector $ U^{1/n} \, {Y}/{\|Y\|_p}$ has distribution $\bU_{n,p}$.}
	\end{itemize}
\end{lemma}%
%
%
%---------------------------------------------------------------------------------------------------------------------------------------------------------------------------------------------------------------------
%
\subsection{LDPs for $q$-norms in $\ell_p^n$-balls}\label{subsec:LpLDPs}
Throughout this paper we assume $1 \le q < p < \infty$. The main variables of interest will be the $q$-norms of the random vectors $Z^{(n)}, \mathscr{Z}^{(n)} \in \B^n_p$ with $Z^{(n)} \sim \bC_{n,p}$ and $\mathscr{Z}^{(n)} \sim \bU_{n,p}$. Note, that we will always denote quantities related to $\mathscr{Z}^{(n)} \sim \bU_{n,p}$ cursively. To get non-trivial results, our target variables also need to be appropriately rescaled. Thus, for random vectors $Z^{(n)}, \mathscr{Z}^{(n)} \in \B^n_p$ with $Z^{(n)} \sim \bC_{n,p}$ and $\mathscr{Z}^{(n)} \sim \bU_{n,p}$, our target  variables will be  $n^{1/p-1/q} \, \|Z^{(n)}\|_q$ and $n^{1/p-1/q} \, \|\mathscr{Z}^{(n)}\|_q$, respectively. We set %
 $$\|Z\| := \left( n^{1/p-1/q} \, \|Z^{(n)}\|_q\right)_{n\in \N} \qquad \text{ and } \qquad \|\mathscr{Z}\| := \left( n^{1/p-1/q} \, \|\mathscr{Z}^{(n)}\|_q\right)_{n\in \N}.$$
It follows via the strong law of large numbers and the continuous mapping theorem applied to the probabilistic representations in \eqref{eq:ProbRepQNormCnp} and \eqref{eq:ProbRepQNormUnp} that the expectations of $\|Z\|$ and $\|\mathscr{Z}\|$ converge in $n \in \N$ to $m_{p,q}:=M_p(q)^{1/q}$.  For fixed $n \in \N$ we will denote %\\
$$
\E \left[n^{1/p-1/q} \, \|Z^{(n)}\|_q\right] := m_{n,p,q} \qquad \text{ and } \qquad \E \left[n^{1/p-1/q} \, \|\mathscr{Z}^{(n)}\|_q\right] := \mathscr{m}_{n,p,q} .
$$
Furthermore, LDPs for $\|Z\|$ and $\|\mathscr{Z}\|$ have been given in previous works, which we want to include here explicitly. But first, let us look at the following probabilistic representations of $\|Z\|$ and $\|\mathscr{Z}\|$, since the LDPs are given with respect to the c.g.f. \hspace{-0.15cm}of these representations: Let $(Y^{(n)})_{n \in\N}$ be a sequence of i.i.d. \hspace{-0.1cm}random vectors $Y^{(n)} := (Y^{(n)}_1, \ldots, Y^{(n)}_n)$ with $Y^{(n)}_i \sim \bN_p$, and $U$ a random variable, independent of the $Y^{(n)}_i$, and uniformly distributed on $[0,1]$. Then, we can see via Lemma \ref{lem:ProbRep} that 
%
%\vspace{-0.425cm} %\text{}\\-----------------------------------------------------
\begin{equation}\label{eq:ProbRepQNormCnp}
n^{1/p-1/q} \, \|Z^{(n)}\|_q  \overset{d}{=} n^{1/p-1/q} \, \frac{\|Y^{(n)}\|_q}{\|Y^{(n)} \|_p} =  \frac{{\left( \frac{1}{n} \sum_{i=1}^n {|Y^{(n)}_i|}^{q}\right)}^{1/q}}{{\left( \frac{1}{n} \sum_{i=1}^n {|Y^{(n)}_i|}^{p}\right)}^{1/p}}, 
\end{equation}
%
%\vspace{-0.425cm} \text{}\\
and 
%
%\text{} \vspace{-0.5cm} \text{}\\
\begin{equation}\label{eq:ProbRepQNormUnp}
 n^{1/p-1/q} \, \|\mathscr{Z}^{(n)}\|_q  \overset{d}{=} n^{1/p-1/q} \, U^{1/n} \, \frac{\|Y^{(n)}\|_q}{\|Y^{(n)} \|_p} = U^{1/n} \,  \frac{{\left( \frac{1}{n} \sum_{i=1}^n {|Y^{(n)}_i|}^{q}\right)}^{1/q}}{{\left( \frac{1}{n} \sum_{i=1}^n {|Y^{(n)}_i|}^{p}\right)}^{1/p}}. 
\end{equation}
%
%\vspace{-0.425cm} \text{}\\
Define 
\begin{equation} \label{eq:Vn}
\hspace{-0.1cm} V^{(n)} :=  \left(V^{(n)}_1, \ldots, V^{(n)}_n \right) \in \R^{2n} \hspace{0.1cm} \qquad \text{ with } \qquad V^{(n)}_i := \left({|Y_i^{(n)}|}^q, {|Y_i^{(n)}|}^p\right), \, \, \quad
\end{equation}
%\vspace{-0.425cm} \text{}\\
and
%\vspace{-0.425cm} \text{}\\
\begin{equation}\label{eq:VnScr}
\quad \mathscr{V}^{(n)} := \left(\mathscr{V}^{(n)}_1, \ldots, \mathscr{V}^{(n)}_n\right) \in \R^{3n} \qquad \text{ with } \qquad \mathscr{V}^{(n)}_i := \left({|Y_i^{(n)}|}^q, {|Y_i^{(n)}|}^p, U^{1/n}\right).
\end{equation}
We denote the m.g.f. \hspace{-0.05cm}and c.g.f. \hspace{-0.05cm}of the $V^{(n)}_i$ as 
\begin{equation} \label{eq:MgfCgf}
 \displaystyle \varphi_p(\tau) := \int_\R e^{tau_1 |y|^q + \tau_2 |y|^p} f_p(y) \, \dint y \qquad \text{ and } \qquad  \Lambda_p(\tau) := \log \int_\R e^{\tau_1 |y|^q + \tau_2 |y|^p} f_p(y) \, \dint y,
\end{equation}
for $\tau=(\tau_1, \tau_2)\in \R^2$ and the Legendre-Fenchel transform of $\Lambda_p$ as 
$$\displaystyle \Lambda_p^*(x) := \sup_{\tau \in \R^2} \big[\langle x, \tau \rangle - \Lambda_p(\tau)\big], \qquad x\in \R^2.$$
Let $\mathcal{D}_p$ be the effective domain of $\Lambda_p$. Since $q<p$, for the integral in both $\varphi_p$ and $\Lambda_p$ to be finite, the sign of the dominant term in the exponent must be negative. Remembering the definition of $f_p$, one can see that this is given for $tau_2 < {1}/{p}$, thus $\mathcal{D}_p = \R \times \left(-\infty, 1/p\right)$. Now, we want to characterize the points $x \in \R^2$ for which there exists a $\tau(x) \in \mathcal{D}_p$, such that %
\begin{equation} \label{eq:ArgSupLegendre}
\Lambda_p^*(x) =  \langle x, \tau(x) \rangle - \Lambda_p(\tau(x)),
\end{equation}
i.e. for which the function $g_x(\tau):=  \langle x, \tau \rangle - \Lambda_p(\tau)$ does not attain its supremum at infinity. We will do this along the lines of \cite[Section 2]{AdrianiBaldi}. It holds that $\Lambda_p$ is convex in $\tau$ (see standard properties of the c.g.f. in e.g. \cite[Lemma 2.2.31]{DZ}), hence $g_x(\tau) := \langle x, \tau \rangle - \Lambda_p(\tau)$ is concave as a sum of concave functions. Then, for a given $x \in \R^2$, there either exists a $\tau(x) \in \R^2$ that satisfies $x = \nabla_\tau \,  \Lambda_p(\tau)$, i.e. that is a root of $\nabla_\tau \,  g_x(\tau)$, or the supremum of $g_x$ will be attained at infinity. If such a $\tau(x)$ exists and lies in $\mathcal{D}_p$, it holds that $\Lambda_p^*(x) =  \langle x, \tau(x) \rangle - \Lambda_p(\tau(x)) <\infty$ (see \cite[Lemma 2.2.31]{DZ}). Since the $V_i^{(n)}$ are not concentrated on a hyperplane (as $1 \le q < p < \infty$), the covariance matrix of their distribution given by $\mathcal{H}_{\tau} \, \Lambda_p(0,0)$ is positive definite and thereby invertible. $\mathcal{H}_{\tau} \, \Lambda_p(\tau)$ for $\tau \in \mathcal{D}_p$ can also be interpreted as the covariance matrix of an exponentially shifted distribution of $V_i^{(n)}$ (see \cite[p. 374]{AdrianiBaldi}), which, by the same argument, is also not concentrated on any hyperplane, hence the $\mathcal{H}_{\tau}\, \Lambda_p(\tau)$ are positive definite as well. This implies that $\Lambda_p(\tau)$ is strictly convex on $\mathcal{D}_p$, thereby also making $g_x$ strictly concave on $\mathcal{D}_p$. Hence, the strict concavity of $g_x$ then ensures that $\tau(x)$ is unique in $\mathcal{D}_p$ in the above described property. We denote the set of $x \in \R^2$ for which such a $\tau(x) \in \mathcal{D}_p$ exists as $\mathcal{J}_p$ and call it the admissible domain of $\Lambda_p^*$.
\begin{rmk} \label{rmk:LegendreBijection}
	Note that the admissible domain $\mathcal{J}_p$ is the image of $\mathcal{D}_p$ under the derivative of the c.g.f. $\Lambda_p$. It actually holds that $\nabla_\tau \Lambda_p(\tau)$ is a bijection from the interior of the effective domain of $\Lambda_p$ into the interior of the effective domain of $\Lambda_p^*$, by the properties of the Legendre transform (see \cite[Theorem 26.5]{Rockafellar}). Since $\mathcal{D}_p$ is open and $\nabla_\tau \Lambda_p$ continuous on $\mathcal{D}_p$, we thereby get that the effective domain of $\Lambda_p^*$ is also open and thus, $\mathcal{J}_p$ is simply the effective domain of $\Lambda_p^*$.
	\end{rmk}
For the sequence $\|Z\|$ the following LDP has already been shown by Kabluchko, Prochno and Thäle \cite[Section 5.1]{KPTLimitThm}: 
\begin{proposition} \label{prop:KPT-LDP-Cnp}
Let $1 \le q < p < \infty$ and $Z^{(n)} \sim \bC_{n,p}$ be a random vector in $\SSS^{n-1}_p$. Then  the sequence $\left( n^{1/p-1/q} \, \|Z^{(n)}\|_q\right)_{n\in \N}$ satisfies an LDP with speed $n$ and good rate function 
$$
\mathcal{I}_{\|Z\|}(z)  := \begin{cases}
\displaystyle \inf_{{\text{$t_1,t_2 >0$}}\atop{\text{$t_1^{1/q}t_2^{-1/p} = z$}}} \Lambda_p^*(t_1,t_2) &:  z > 0\\
+\infty&: z \le 0.
\end{cases}
$$
\end{proposition}
 In \cite[Lemma 2.1, Appendix A]{LiaoRamanan} Liao and Ramanan established a simplification of a similar rate function in a different setting. Their arguments can be analogously applied in our setting to derive the following result:
\begin{lemma}\label{lem:UniqueMinRF}
	Let $z > m_{p,q}$ such that $z^* := (z^q, 1) \in \mathcal{J}_p$. Then
	$$\displaystyle \mathcal{I}_{\|Z\|}(z)  = \inf_{{\text{$t_1,t_2 >0$}}\atop{\text{$t_1^{1/q}t_2^{-1/p} = z$}}}  \Lambda_p^*(t_1, t_2) = \Lambda_p^*(z^*),$$
	with $z^{*}$ being the unique point at which $\Lambda _{p}^{*}$ attains its infimum under the above conditions.
\end{lemma} 
To keep this paper self-contained, we will present the analogous proof of this in the Appendix. %
For the sequence $\|\mathscr{Z}\|$, the following LDP was also provided by Kabluchko, Prochno and Thäle in \cite[Theorem 1.2]{KPTLimitThm}:%
\begin{proposition} \label{prop:KPT-LDP-Unp}
Let $1 \le q < p < \infty$ and $\mathscr{Z}^{(n)} \sim \bU_{n,p}$ be a random vector in $\B^n_p$. Then  the sequence $\left( n^{1/p-1/q} \, \|\mathscr{Z}^{(n)}\|_q\right)_{n \in \N}$ satisfies an LDP with speed $n$ and good rate function 
$$
\mathcal{I}_{\|\mathscr{Z}\|}(z) := \begin{cases}
\displaystyle \inf_{{\text{$z = z_1 z_2$}}\atop{\text{$z_1, z_2 >0$}}} \big[ \mathcal{I}_{\|Z\|}(z_1) + \mathcal{I}_U(z_2)\big] &:  z > 0\\
+\infty&: z\le0,
\end{cases}
$$
with $\mathcal{I}_{\|Z\|}$ as in Proposition \ref{prop:KPT-LDP-Cnp} and 
$$\mathcal{I}_{U}(z_2) := \begin{cases}
\displaystyle - \log(z_2) &:  z_2 \in (0,1]\\
+\infty&: otherwise.
\end{cases}
$$
\end{proposition}
We again show that the above infimum is attained at a unique point satisfying the infimum condition. 
 \begin{lemma}\label{lem:UniqueMinRFUnp}
 Assume the same setting as in Proposition \ref{prop:KPT-LDP-Unp}. For $z> m_{p,q}$, we can simplify the rate function by combining the two infimum operations to get 
 $$ \mathcal{I}_{\|\mathscr{Z}\|}(z) =  \displaystyle \inf_{{\text{$z = t_1^{1/q} t_2^{-1/p}t_3$}}\atop{\text{$t_1,t_2 >0, t_3 \in (0,1]$}}} \big[ \Lambda_p^*(t_1, t_2) -  \log(t_3)\big].$$
We define 
$$\mathcal{I}_{\mathscr{S}}(t) := \big[ \Lambda_p^*(t_1, t_2) -  \log(t_3)\big], \qquad t_1, t_2 \in \R, \quad t_3 \in (0,1],$$
and set $z^* := (z^q,1) \in \R^2$,  $z^{**} := (z^q, 1, 1) \in \R^3$. It then holds for $ z > m_{p,q}$ with $z^{*} \in  \mathcal{J}_p$ that 
$$\displaystyle \mathcal{I}_{\|\mathscr{Z}\|}(z)  =  \mathcal{I}_{\mathscr{S}}(z^{**}) = \Lambda_p^*(z^*),$$
with $z^{**}$ being the unique point at which $\mathcal{I}_{\mathscr{S}}$ attains its infimum under the above conditions.
\end{lemma} 
Thus, for $z > m_{p,q}$ with $z^* \in \mathcal{J}_p$ both $\|Z\|$ and $\|\mathscr{Z}\|$ satisfy LDPs with the same speed and rate function. Again, the proof of this is relegated to the Appendix.
\begin{rmk}
Note that in the results within this paper, deviations from the "limit expectation" $m_{p,q}$ are considered, even though the sequences $\|Z\|$ and $\|\mathscr{Z}\|$ have respective expectations $m_{n,p,q}$ and $\mathscr{m}_{n,p,q}$, that only converge to $m_{p,q}$ in $n\in \N$. This, however, is not an issue for our results. As seen in \eqref{eq:ProbRepQNormCnp} and \eqref{eq:ProbRepQNormUnp}, the sequences are represented via the empirical averages of probabilistic representations seen in \eqref{eq:Vn} and \eqref{eq:VnScr}. The expectations of these representations only ever play a role in our proofs regarding the behaviour of the corresponding c.g.f.s, specifically only in the case of $\|Z\|$ (e.g. in the proofs of Lemma \ref{lem:UniqueMinRFUnp} and Lemma \ref{lem:UniqueMinCl} or implicitly in the proof of the density approximations in Section \ref{sec:JointDensityEstimate}). As the $V_i^{(n)}$ in \eqref{eq:Vn} are i.i.d., they all share the same c.g.f. as given in \eqref{eq:MgfCgf} and the same expectation $\E[V_i^{(n)}] = (M_p(q), M_p(p)) = (m_{p,q}^q, 1)$. Hence, the fact that the expectation $m_{n,p,q}$ only converges to $m_{p,q}$ does not affect our proofs.  This is in keeping with classical results from large deviations theory like the Theorem of Gärtner-Ellis (see \cite[Theorem V.6]{dH}), where an arbitrary (i.e. not necessarily i.i.d.) sequence of random variables is not required to have a shared expectation, but rather that the sequence of the (appropriately rescaled) c.g.f.s of the individual random variables in the sequence converge to a fixed function with the origin in the interior of its effective domain. The resulting LDP then considers deviation probabilities from the limiting expectation as well. In the case of $\|\mathscr{Z}\|$ the c.g.f.s of the $\mathscr{V}_i^{(n)}$ are not employed at all (neither themselves nor their limit in $n$). Instead, we simply use the density approximation in Proposition \ref{prop:DensityEstimate} for the sum of the $V_i^{(n)}$ and make use of the independence of $U^{1/n}$ from the coordinates of $V_i^{(n)}$. Since our main results assume $n \in \N$ to be sufficiently large (that is, large enough for the local density approximations in Section \ref{sec:JointDensityEstimate} to hold), this effectively means that for $n\in\N$ sufficiently large, the difference of $m_{p,q}$ and $m_{n,p,q}, \mathscr{m}_{n,p,q}$ is of order at most $o(1)$ and therefore does not affect our SLD estimates. %
\end{rmk}
%
%---------------------------------------------------------------------------------------------------------------------------------------------------------------------------------------------------------------------
%
%
%
\subsection{A few remarks on Weingarten maps and curvature}\label{subsec:Weingarten}
As outlined in the introduction, we will finish the proof of our first main result in Theorem \ref{thm:MainResult1} by integrating over a previously established density estimate via a result of Adriani and Baldi \cite{AdrianiBaldi} for Laplacian integral expansions. This result has a heavily geometric flavour and relies on the Weingarten maps of certain hypersurfaces, which in our case are simply curves in $\R^2$. We will therefore just give a brief reminder of the Weingarten map in this setting, recall some of its properties, and refer to the relevant literature (e.g. \cite{Hicks, Klingenberg}) or Adriani and Baldi \cite {AdrianiBaldi} for a more in-depth discussion of the topic.\\ 
\\
In general, the Weingarten map of a smooth hypersurface $M \subset \R^d$ at a point $p \in M$ is an endomorphism of the tangent space $T_pM$ at $p$, mapping any $y \in T_pM$ to the directional derivative of a normal field of $M$ in $p$ in the direction of $y$.  However, as remarked in \cite[Example 4.3]{AdrianiBaldi}, for $d=2$, hypersurfaces simplify to planar curves and the Weingarten map at a point $p$ simplifies to the absolute value of the curvature $K(p)$ of the curve at $p$. For implicit curves, i.e. curves given as the zero set of a function, we have the following formula for its curvature from \cite[Proposition 3.1]{Goldman2005}: 
 \begin{lemma} \label{lem:ImplCurves}
Let $F: \R^2 \to \R$ be a smooth function. For a curve $\mathscr{C}:= \{x \in \R^2: F(x) =0\}$ given as the zero set of $F$, and a point $p \in \mathscr{C}$, where $\nabla_{x} F(p) \ne 0$, it then holds that
 $$K(p) = \displaystyle \frac{ \left(-F_{[0,1]}, F_{[1,0]} \right) \left(\begin{array}{rrr} 
F_{[2,0]} & F_{[1,1]} \\ 
F_{[1,1]} & F_{[0,2]}  \\ 
\end{array}\right) 
\left(-F_{[0,1]}, F_{[1,0]}\right)
}{{\left({F_{[1,0]}}^2 + {F_{[0,1]}}^2\right)}^{3/2}},$$
with derivatives $F_{[i,j]} = F_{[i,j]}(p)$ as in \eqref{eq:AblNotation}.
 \end{lemma}
\begin{rmk} \label{rmk:Curvature}\text{}%\\%
%\vspace{-0.5cm}
\begin{itemize}%
\item[i)]{Given the set-up of the previous Lemma, straightforward calculation of the above fraction gives that
$$K(p) = \displaystyle \frac{{F_{[0,1]}}^2{F_{[2,0]}} - 2{F_{[0,1]}}{F_{[1,0]}}{F_{[1,1]}} + {F_{[1,0]}}^2{F_{[0,2]}}}{{\left({F_{[1,0]}}^2 + {F_{[0,1]}}^2\right)}^{3/2}}.$$ }
\item[ii)]{In case that $\mathscr{C}$ is the graph of a smooth function $f:\R \to \R$, i.e. $\mathscr{C} = \{(x_1, x_2) \in \R^2: x_2 = f(x_1)\}$, and $p=(x, f(x))$, the above reduces to 
$$K(p) = \displaystyle \frac{|f^{\prime\prime}(x)|}{{\big(1 + f^\prime(x)^2\big)}^{3/2}}.$$}
\end{itemize}
\end{rmk}
\text{}
%
%
%
%---------------------------------------------------------------------------------------------------------------------------------------------------------------------------------------------------------------------
%
%
\section{Main Results} \label{sec:MainResults}
Using the concepts and notation established in the previous section, we now proceed to present our main results and their applications:
\subsection{Sharp asymptotics for $q$-norms of random vectors in $\SSS_p^{n-1}$ and $\B_p^n$} \label{subsec:MainResults}
For $Z^{(n)} \sim \bC_{n,p}$, we want to give sharp asymptotics for the probability $\Pro\left(n^{1/p-1/q} \, \|Z^{(n)}\|_q> z\right)$ for $z > m_{p,q}$ such that $z^* \in \mathcal{J}_p$, with $z^*$ as  defined in Lemma \ref{lem:UniqueMinRF}. Before presenting our results, let us define the deviation-dependent functions $\xi(z)$ and $\kappa(z)$, as mentioned also in the sharp large deviation results of Bahadur and Ranga Rao \cite{Bahadur}. %
  For $x \in \R^2$, we set 
 \begin{equation} \label{eq:HessX}
\displaystyle \mathfrak{H}_{x} :=  \mathcal{H}_{\tau}\Lambda_p(\tau(x))
\end{equation}
to be the Hessian of the c.g.f. \hspace{-0.05cm}$\Lambda_p(\tau)$ in $\tau \in \R^2$, evaluated at $\tau(x)$. %
For $z > m_{p,q}$ such that $z^* \in \mathcal{J}_p$,  we then define the deviation-dependent functions as 
\begin{equation}\label{eq:NormTermXi}
\displaystyle \xi(z)^2 := \langle \mathfrak{H}_{z^*} \, \tau(z^*), \tau(z^*) \rangle\, \det \mathfrak{H}_{z^*},
\end{equation}
\begin{equation}\label{eq:NormTermKappa}
\displaystyle \kappa(z)^2 := 1 - \frac{\left(\tau(z^*)_1^2 + \tau(z^*)_2^2\right)^{3/2} \, |pq(p-q)z^q|}{\big{|}\tau(z^*)_2^2 \left(\mathfrak{H}_{z^*}^{-1}\right)_{11} - 2 \tau(z^*)_1  \tau(z^*)_2 \left(\mathfrak{H}_{z^*}^{-1}\right)_{12} + \tau(z^*)_1^2 \left(\mathfrak{H}_{z^*}^{-1}\right)_{22} \big{|} \, (z^{2q} + p^2q^{-2})^{3/2}}.
%\vspace{0.125cm}
\end{equation}
%
%\text{}\\
%
%
\vspace{-0.025cm}
\begin{thm}\label{thm:MainResult1}
Let $1 \le q < p < \infty$, $n \in \N$, and $Z^{(n)}$ be a random vector in $\B^n_p$ with $Z^{(n)} \sim \bC_{n,p}$. Then, for $n$ sufficiently large and any $z > m_{p,q}$ such that $z^* \in \mathcal{J}_p$, it holds that %
$$
	\Pro\left(n^{1/p-1/q} \|Z^{(n)}\|_q > z\right) = \frac{1}{\sqrt{2\pi n} \, \kappa(z) \xi(z)} \, e^{-n \, \Lambda_p^*(z^*)} \, (1 + o(1)).
$$
\end{thm}
We want to do the same for $\Pro\left(n^{1/p-1/q} \, \|\mathscr{Z}^{(n)}\|_q> z\right)$ with $\mathscr{Z}^{(n)} \sim \bU_{n,p}$ and $z > m_{p,q}$. Again, we start by defining our deviation-dependent function for $z >m_{p,q}$ \vspace{0.125cm}
\begin{eqnarray}\label{eq:NormTermGamma}
\nonumber\gamma(z)^2 &:=& \det \mathfrak{H}_{z^*} \,  \tau(z^*)_1^{2} \, (qz^q \tau(z^*)_1+1)^{2} \vphantom{\sum}\\ 
%\nonumber \\
&& \times \, \Bigg[  \frac{z^{2q} q^2}{p^2}\left(\mathfrak{H}_{z^*}^{-1}\right)_{11} + \frac{2z^q q}{p} \left(\mathfrak{H}_{z^*}^{-1}\right)_{12} +  \left(\mathfrak{H}_{z^*}^{-1}\right)_{22} + \, \tau(z^*)_1 \frac{z^q q(q-p)}{p^2} \Bigg]. 
\end{eqnarray}%
\begin{thm}\label{thm:MainResult2}
Let $1 \le q < p < \infty$, $n \in \N$, and $\mathscr{Z}^{(n)}$ be a random vector in $\B^n_p$ with $\mathscr{Z}^{(n)} \sim \bU_{n,p}$. Then, for $n$ sufficiently large and any $z > m_{p,q}$ such that $z^* \in \mathcal{J}_p$, it holds that %
$$
\Pro\left(n^{1/p-1/q} \|\mathscr{Z}^{(n)}\|_q > z\right) = \frac{1}{\sqrt{2\pi n} \, \gamma(z)} \, e^{-n \, \Lambda_p^*(z^*)} \, (1 + o(1)).
$$
\end{thm}
We have seen in Section \ref{subsec:LpLDPs} that $\|Z\|$ and $\|\mathscr{Z}\|$ both satisfy LDPs with the same speed and rate function for $z > m_{p,q}$ such that $z^* \in \mathcal{J}_p$, despite the underlying distributions being different. Comparing Theorem \ref{thm:MainResult1} and Theorem \ref{thm:MainResult2} now paints a different picture, with the sharp asymptotics for $\|Z\|$ and $\|\mathscr{Z}\|$ being noticeably different. As mentioned in our introductory statements, idiosyncratic phenomena of underlying distributions, which can be drowned out on the LDP scale, are often still visible on the scale of sharp large deviations. This is in keeping with what was shown in \cite[Theorem 2.4, Theorem 2.6]{LiaoRamanan} for one-dimensional projections of $\ell_p^n$-spheres and $\ell_p^n$-balls. 
\begin{rmk} 
Let us draw a brief comparison between our results and the concentration inequality that follows by the Gromov-Milman Theorem as discussed in \cite[Remark, p. 1062]{NaorTAMS}. Therein, it is shown that the Gromov-Milman theorem from \cite{Gromov} implies that for $1 < q \le p < \infty$ and a random vector $Z^{(n)} \sim \bC_{n,p}$, it holds that 
$$\Pro\Big( \big| n^{1/p - 1/q}\|Z^{(n)}\|_q -  m_{n,p,q} \big| \ge z \Big) \le C \, \exp \left(-c \, n \, z^{\max\{2,p\}} \right),$$
where $C>0$ and $c>0$ are constants. If we consider the set-up of Theorem \ref{thm:MainResult1}, i.e. $1 \le q < p < \infty$ and $z > m_{n,p,q}$, and only consider deviations without the absolute value, we can derive from the above that
$$\Pro\Big( n^{1/p - 1/q}\|Z^{(n)}\|_q  > z \Big) \le C \, \exp \left(-c \, n \, z^{\max\{2,p\}} \right). \vspace{0.2cm}$$
Comparing this with our sharp large deviation results from Theorem \ref{thm:MainResult1} for $z > m_{p,q}$ such that $z^* \in \mathcal{J}_p$,
$$
	\Pro\left(n^{1/p-1/q} \|Z^{(n)}\|_q > z\right) = \frac{1}{\sqrt{2\pi n} \, \kappa(z) \xi(z)} \, e^{-n \, \Lambda_p^*(z^*)} \, (1 + o(1)),
$$
we can see that our results improve on the estimate in terms of $n\in\N$ by a factor of $n^{-1/2}$ and give explicit and deviation-dependent terms $\kappa(z)$ and $\xi(z)$ instead of fixed constants for all deviations $z$.
\end{rmk}
\begin{rmk} \label{rmk:CompLR} When comparing the SLD results in Theorem~\ref{thm:MainResult1} and Theorem~\ref{thm:MainResult2} to those of Liao and Ramanan \cite[Theorem 2.4, Theorem 2.6]{LiaoRamanan}, one directly notices the core difference in the settings. Liao and Ramanan examine projections of random vectors on ${\mathbb{S}}^{n-1}_{p}$ and ${\mathbb{B}}^{n}_{p}$ with respective distributions $\mathbf{C}_{n,p}$ and $\mathbf{U}_{n,p}$ onto fixed one-dimensional subspaces, and therefore have to consider weighted sums of \textit{dependent} random vectors as probabilistic representations. Thus, all their results have to be conditioned on the projection space and include additional terms accounting for the specifics of the subspace. In our case however, the probabilistic representations are given as sums of \textit{i.i.d.} random variables (see Section~\ref{sec:ProbRep}), which does not necessitate these additional factors. Therefore, when using results from Liao and Ramanan \cite{LiaoRamanan}, we adapt their usage accordingly to the given probabilistic representations in our setting. Beyond that however, the SLD results share several similarities, especially when comparing the deviation-dependent terms $\kappa $, $\xi $ and $\gamma $, which for $q=1$ are almost equal.
\end{rmk}
 Both proofs of Theorem~\ref{thm:MainResult1} and Theorem~\ref{thm:MainResult2} contain three essential steps, as already briefly mentioned in the introduction. The first will be rewriting the probabilities in both theorems with respect to convenient probabilistic representations, specifically $S^{(n)}$ and $\mathscr{S}^{(n)}$ given in~\eqref{eq:PropRepVector} of Section~\ref{sec:ProbRep} as the respective empirical averages of the $V^{(n)}_{i}$ and $\mathscr{V}^{(n)}_{i}$ in~\eqref{eq:Vn} and~\eqref{eq:VnScr}. The idea is to write the deviation probabilities as an integral of their distribution over a given ``deviation area''. The second step is giving local density approximations for these representations. Since the entries of both the $V^{(n)}$ and the $\mathscr{V}^{(n)}$ are highly dependent, no canonical joint densities are available to us to easily do so. However, their Fourier transforms can be given explicitly, thus, for $n \in \mathbb{N}$ large enough one can use the Fourier inversion theorem to write the densities of $S^{(n)}$ and $\mathscr{S}^{(n)}$ as integrals of their Fourier transforms. Heuristically speaking, this means that while the individual $V^{(n)}_{i}$ and $\mathscr{V}^{(n)}_{i}$ do not possess densities in $\mathbb{R}^{2}$, but for $n \in \mathbb{N}$ sufficiently large their empirical averages $S^{(n)}$ and $\mathscr{S}^{(n)}$ asymptotically do. The resulting integrals can then be approximated using the saddle point method. Since our representations are given as sums of i.i.d. random vectors, for whom this has been done in previous results (see, e.g., \cite{BorovkovRogozin,Daniels1954,Richter1957,Richter1958}), we will not prove the density approximations here explicitly. The third and final step then is to calculate the integrals of these densities over their respective deviation area. For $\|Z\|$, this is done by a result of Adriani and Baldi \cite{AdrianiBaldi}, which construes the boundary of the deviation area and the level sets of the rate function in the corresponding LDP as hypersurfaces, which are just planar curves in our setting, and uses their Weingarten maps to approximate the integral. For $\|\mathscr{Z}\|$, this is not applicable, as certain differentiability conditions are no longer met. Thus, a result by Breitung and Hohenbichler \cite{Breitung1989} is used, which allows for multi-dimensional Laplace integral approximations under less restrictive differentiability conditions.
\subsection{Intersection volumes of $\ell_p^n$-balls}\label{subsec:Application}
We want to use our sharp large deviation results to further the findings of Schechtman and Schmuckenschläger \cite{SchechtmanSchmuckenschlaeger} and Schmuckenschläger \cite{SchmuckCLT} for intersection volumes of $t$-multiples of different $\ell_p^n$-balls. We will first give a brief overview of the original results. For $p\in [1, \infty)$, we define $\D_p^n:= \vol_n(\B_p^n)^{-1/n} \, \B_p^n$ to be the volume normalized $\ell_p^n$-ball and recall that 
$$\vol_n(\B_p^n) = \frac{{\left(2 \Gamma \left(1 + \frac{1}{p}\right)\right)}^n}{\Gamma\left(1+\frac{n}{p}\right)}.$$
We furthermore set 
$$c_{n,p} := n^{1/p} \, \vol_n\left(\B_p^n\right)^{1/n} \qquad \text{ and } \qquad c_p:= 2 \, e^{1/p} \, p^{1/p} \, \Gamma\left(1 + \frac{1}{p}\right),$$
 and recall that it was shown in \cite{SchechtmanSchmuckenschlaeger} that $\lim \limits_{n\to \infty} c_{n,p} = c_p$. Moreover, for $p,q \in[1, \infty), p\ne q$, we set 
 $$c_{n, p,q} := \frac{c_{n,p}}{c_{n,q}}, \qquad \qquad A_{n,p,q}:= \frac{c_{n,p}}{m_{p,q} \, c_{n,q}}, \qquad \text{ and } \qquad A_{p,q}:= \lim_{n \to \infty} A_{n,p,q}. $$
 Hence, it follows that
% \vspace{-0.25cm}
 $$A_{p,q}= \frac{c_{p}}{m_{p,q} \,  c_{q}} = \frac{ \Gamma \left(1 + \frac{1}{p}\right)^{1 + (1/q)}}{ \Gamma \left(1 + \frac{1}{q}\right)  \Gamma\left(\frac{q+1}{p} \right)^{1/q}} \, \, e^{1/p - 1/q}.$$
Lastly, for $t \ge 0$ and $n\in\N$, we define $t_n \ge0$ such that 
$$t_n \frac{A_{p,q}}{A_{n,p,q}} = t.$$
Having established the necessary notation, we shall now recall the result of Schmuckenschläger \cite[Theorem 3.3]{SchmuckCLT}. Therein, it was shown that for $p,q \in [1, \infty), p\ne q$, and $t\ge0$ it holds that 
\begin{equation}\label{eq:VolIntersecConv}
\vol_n \left(\D_p^n \cap t\D_q^n\right) \underset{n\to \infty}{\longrightarrow} \begin{cases}
1 &: A_{p,q}\, t > 1\\
\frac{1}{2} &: A_{p,q}\, t = 1\\
0 &: A_{p,q}\, t < 1.
\end{cases}
\end{equation}
%\text{}\\
%
To prove this, a central limit theorem for $n^{1/p - 1/q} \|\mathscr{Z}^{(n)}\|_q$ with $\mathscr{Z}^{(n)} \sim \bU_{n,p}$ and $p,q \in [1, \infty)$, $p \ne q$, is shown in \cite[Proposition 2.4, Proof of Theorem 3.2]{SchmuckCLT}, since $\vol_n(\D_p^n \cap t\D_q^n)$ can be written as
% \vspace{-0.25cm}
\begin{eqnarray} \label{eq:VolIntersec}
\nonumber \vol_n(\D_p^n \cap t\D_q^n)&=& \vol_n\left( \left\{ z \in \D_p^n: z \in  t_n \, \frac{A_{p,q}}{A_{n,p,q}} \, \D_q^n \right \} \right) \vphantom{\int_0}\\
%\nonumber \\
\nonumber&=& \vol_n\left( \left\{ z \in \D_p^n: z \in  t_n \, A_{p,q} \,m_{p,q} \, \frac{c_{n,q}}{c_{n,p}} \, \D_q^n \right \} \right)\vphantom{\int_0}\\
%\nonumber \\
\nonumber &=&  \vol_n\left( \left\{ z \in \vol_n(\B_p^n)^{-1/n} \, \B_p^n: z \in  t_n \,  A_{p,q} \,m_{p,q} \, n^{1/q - 1/p} \, \vol_n(\B_p^n)^{-1/n} \, \B_q^n \right \} \right)\vphantom{\int_0}\\
%\nonumber \\
\nonumber &=& \vol_n(\B_p^n)^{-1} \,  \vol_n\left( \left\{ z \in \B_p^n: z \in  t_n \,  A_{p,q} \,m_{p,q} \, n^{1/q - 1/p} \, \B_q^n \right \} \right)\vphantom{\int_0}\\
%\nonumber\\
 &=& \Pro \left( n^{1/p - 1/q} \|\mathscr{Z}^{(n)}\|_q \le t_n \, A_{p,q} \,m_{p,q} \right).
\end{eqnarray}
However, we know from the Berry-Esseen Theorem (see \cite[Theorem 2.1.3]{Vershynin2018}) that the error of the Gaussian approximation given by a central limit theorem decreases with rate $n^{-1/2}$. Thus, using \eqref{eq:VolIntersec} and the central limit theorem from \cite{SchmuckCLT}, we can only infer a rate of convergence of $n^{-1/2}$ in \eqref{eq:VolIntersecConv}. Using Theorem \ref{thm:MainResult2}, we can considerably refine that rate of convergence in the first of the three cases in \eqref{eq:VolIntersecConv} from a sublinear rate to an exponential rate for $1 \le q<p<\infty$.
\begin{proposition}
Let $1 \le q < p < \infty$ and $n\in\N$. Using the notation established above, it then holds for $t > m_{p,q} \, {c_{n,p,q}}^{-1}$ such that $(t \,  c_{n,p,q})^* \in \mathcal{J}_p$, and sufficiently large $n \in \N$ that 
%
%\vspace{-0.2cm}
 $$\vol_n\left(\D_p^n \cap t\D_q^n\right) = 1 - \frac{1}{\sqrt{2\pi n} \, \gamma( t \, c_{n,p,q})} \, e^{-n \, \Lambda_p^*(( t\,  c_{n,p,q})^*)} \, (1 + o(1)).$$
%\vspace{-0.3cm}
%
\end{proposition}
\begin{proof}
Let $1 \le q < p < \infty$, $t > m_{p,q} \, {c_{n,p,q}}^{-1}$ such that $(t \,  c_{n,p,q})^* \in \mathcal{J}_p$, and assume $\mathscr{Z}^{(n)}$ is a random vector in $\B_p^n$ with $\mathscr{Z}^{(n)} \sim \bU_{n,p}$. Using \eqref{eq:VolIntersec}, we get that 
\begin{equation}
\nonumber \vol_n\left(\D_p^n \cap t\D_q^n\right)= \Pro \left( n^{1/p - 1/q} \|\mathscr{Z}^{(n)}\|_q \le t_n A_{p,q} \,m_{p,q} \right) = 1 - \Pro \left( n^{1/p - 1/q} \|\mathscr{Z}^{(n)}\|_q > t_n A_{p,q} \,m_{p,q}  \right).
\end{equation}
%\text{}\\
It now holds that, by $t> m_{p,q} \, {c_{n,p,q}}^{-1}$, we have that  $t \, m_{p,q}^{-1} \, c_{n,p,q} = t \, A_{n,p,q} = t_n \, A_{p,q} > 1$, and hence $t \, c_{n,p,q} = t_n \, A_{p,q} \,m_{p,q} > m_{p,q}$ with $(t \, c_{n,p,q})^* \in \mathcal{J}_p$. Thus, by Theorem \ref{thm:MainResult2}, it follows that 
 $$\vol_n\left(\D_p^n \cap t\D_q^n\right) = 1 - \frac{1}{\sqrt{2\pi n} \, \gamma( t \, c_{n,p,q})} \, e^{-n \, \Lambda_p^*( (t \,  c_{n,p,q} )^*)} \, (1 + o(1)),$$
 which finishes our proof.
\end{proof}
\subsection{One-dimensional projections of $\ell_q^n$-balls}\label{subsec:Application2}
In Remark \ref{rmk:CompLR} we have already discussed the differences between the setting of the results of Liao and Ramanan \cite{LiaoRamanan} and the setting of this paper. However, a geometrically similar result to those in \cite{LiaoRamanan} follows from Theorem \ref{thm:MainResult1}. In \cite[Section 2.4 ]{KPTLimitThm} Kabluchko, Prochno and Thäle derived a central limit theorem for the the length of the projection of an $\ell_p^n$-ball onto the line spanned by a random vector $\theta^{(n)} \in \SSS^{n-1}$ with $\theta^{(n)} \sim \bC_{n,2}$ as a corollary of their main results. We will proceed similarly and derive sharp large deviation results in the same setting. To be specific, in \cite{LiaoRamanan} sharp asymptotics where provided for the \textit{scalar product} of a random vector $Z^{(n)} \sim \bC_{n,p}$ on $\SSS^{n-1}_p$ with a random vector $\theta^{(n)} \sim \bC_{n,2}$ on $\SSS^{n-1}$, which can be negative. We, on the other hand, consider the \textit{absolute value} of the scalar product of such random vectors, thereby only considering non-negative values. \\
\\
In the following, for $q \in [1, \infty]$, define its conjugate $q^*$ via $ {1}/{q} + {1}/{q^*} = 1$, setting ${1}/{\infty} =0$ by convention. Furthermore, for a vector $\theta^{(n)} \in \SSS^{n-1}$, we write $P_{\theta^{(n)}} \B_q^n$ for the projection of $\B_q^n$ onto the line spanned by $\theta^{(n)}$. Then, our quantity of interest is the projection length $\vol_1 \hspace{-0.05cm}\left(P_{\theta^{(n)}} \B_q^n\right)$.

\begin{cor}
Let $2 < q \le \infty$ and $\theta^{(n)} \in \SSS^{n-1}$ be a random vector with $\theta^{(n)} \sim \bC_{n,2}$. Then, for any $z > 2 \, m_{2,q^*}$ such that $(\frac{z}{2})^* \in \mathcal{J}_p$, and sufficiently large $n \in \N$, it holds that
$$
	\displaystyle \Pro\left(n^{1/2-1/q} \, \vol_1 \hspace{-0.05cm}\left(P_{\theta^{(n)}} \B_q^n\right) > z\right) = \frac{1}{\sqrt{2\pi n} \, \kappa\left(\frac{z}{2}\right) \xi\left(\frac{z}{2}\right)} \, e^{-n \, \Lambda^*_2({(\frac{z}{2})}^*)} \, (1 + o(1)),
$$
with $\Lambda_2$ as in \eqref{eq:MgfCgf} and $\xi, \kappa$ as in \eqref{eq:NormTermXi}, \eqref{eq:NormTermKappa}, respectively, for $q^*$ and  $p = 2$.
\end{cor}
\begin{proof}
It holds that 
$$
 \Pro\left(n^{1/2-1/q} \, \vol_1 \hspace{-0.05cm}\left(P_{\theta^{(n)}} \B_q^n\right) >   z \right) =  \Pro\left(n^{1/2-1/q} \, 2 \sup_{x \in \B_{q}^n} \hspace{-0.1cm}|\langle x, \theta^{(n)} \rangle|  > z\right) = \Pro\left(n^{1/2-1/q} \, \|\theta^{(n)}\|_{q^*}  > \frac{z}{2}\right).
$$
Since $2 < q \le \infty$, we have $ 1\le q^* <2 =p$, whereby we can apply Theorem \ref{thm:MainResult1} to the above to get that
$$\Pro\left(n^{1/2-1/q} \, \vol_1 \hspace{-0.05cm}\left(P_{\theta^{(n)}} \B_q^n\right) > z \right) = \frac{1}{\sqrt{2\pi n} \, \kappa\left(\frac{z}{2}\right) \xi\left(\frac{z}{2}\right)} \, e^{-n \, \Lambda^*_2\left((\frac{z}{2})^*\right)} \, (1 + o(1)),$$
with $\Lambda_2, \xi, \kappa$ as described above, which concludes our proof.
\end{proof}

\text{}
\section{Probabilistic Representation} \label{sec:ProbRep}
Recalling the definitions of the random vectors $V^{(n)}$ and $\mathscr{V}^{(n)}$ from \eqref{eq:Vn} and \eqref{eq:VnScr}, we define 
\begin{equation} \label{eq:PropRepVector}
S^{(n)} := \frac{1}{n} \sum_{i=1}^n  V^{(n)}_i  \qquad \text{ and } \qquad \mathscr{S}^{(n)} := \frac{1}{n} \sum_{i=1}^n  \mathscr{V}^{(n)}_i
\end{equation}
as the empirical averages of their respective coordinates. Furthermore, we define the sets 
$$
D_{z}:= \{(t_1, t_2) \in \R^2: t_1,t_2 > 0, t_1^{1/q} \, t_2^{-1/p} > z\},
$$
and 
$$
\mathscr{D}_{z}:= \{(t_1, t_2, t_3) \in \R^3: t_1,t_2 > 0, \,  t_3 \in (0,1], \, t_3 \, t_1^{1/q} \, t_2^{-1/p} > z \}.
$$
It then follows from the reformulations of $\|Z^{(n)}\|_q$ and $\|\mathscr{Z}^{(n)}\|_q$ in  \eqref{eq:ProbRepQNormCnp} and \eqref{eq:ProbRepQNormUnp} that we can write the probabilities within Theorem \ref{thm:MainResult1} and Theorem \ref{thm:MainResult2} with respect to $S^{(n)}$ and $\mathscr{S}^{(n)}$, respectively, as
\begin{equation}\label{eq:ProbRepDev}
\Pro \left(n^{1/p-1/q} \, \|Z^{(n)}\|_q > z\right) = \Pro\left( \frac{1}{n} \sum_{i=1}^n |Y^{(n)}_i|^{q} > z^q \, {\left( \frac{1}{n} \sum_{i=1}^n |Y^{(n)}_i|^{p}\right)}^{\frac{q}{p}} \right) = \Pro\left(S^{(n)} \in D_{z}\right),
\end{equation}
and 
\begin{equation}\label{eq:ProbRepDevUnp}
\Pro\left(n^{1/p-1/q} \, \|\mathscr{Z}^{(n)}\|_q > z\right) = \Pro\left(  U^{\frac{q}{n}} \, \frac{1}{n} \sum_{i=1}^n |Y^{(n)}_i|^{q} > z^q  \, {\left( \frac{1}{n} \sum_{i=1}^n |Y^{(n)}_i|^{p}\right)}^{\frac{q}{p}} \right) = \Pro\left(\mathscr{S}^{(n)} \in \mathscr{D}_{z}\right).
\end{equation}
We refer to these sets as ``deviation areas'', as $S^{(n)}$ or  $\mathscr{S}^{(n)}$ lying in $D_z$ or $\mathscr{D}_z$ represents a deviation of  $\|Z^{(n)}\|_q$ and $\|\mathscr{Z}^{(n)}\|_q$. Note that the boundaries of the deviation areas
$$\partial D_{z} = \{(t_1, t_2) \in \R^2: t_1,t_2 > 0,  t_1^{1/q} \, t_2^{-1/p} = z\}$$
and
$$\partial \mathscr{D}_{z}  = \{(t_1, t_2, t_3) \in \R^3: t_1,t_2 >0, \,  t_3 \in (0,1], \, t_3 \,  t_1^{1/q} \, t_2^{-1/p} = z \}$$
\text{}\vspace{-0.3cm}\\
are the same sets given by the infimum conditions in the respective LDPs for $\|Z\|$ and $\|\mathscr{Z}\|$ in Proposition \ref{prop:KPT-LDP-Cnp} and Proposition \ref{prop:KPT-LDP-Unp}. The fact that for $z > m_{p,q}$, the rate functions of these LDPs both assume a unique minimum on $\partial D_{z}$ and $\partial \mathscr{D}_{z}$, respectively, as was shown in Lemma \ref{lem:UniqueMinRF} and Lemma \ref{lem:UniqueMinRFUnp}, will be essential to the proof of our main results in Sections \ref{sec:ProofMainResult1} and \ref{sec:ProofMainResult2}. We can expand this unique infimum property onto the entirety of $\overline{D_{z}}$ and $\overline{\mathscr{D}_{z}}$, as the following lemma will show:
\begin{lemma}\label{lem:UniqueMinCl}
Assume the same set-up as in Lemma \ref{lem:UniqueMinRF} and Lemma \ref{lem:UniqueMinRFUnp}. Let $z > m_{p,q}$ such that $z^* \in \mathcal{J}_p$. Then 
\begin{enumerate}
\item[i)]{$z^* = (z^q,1)$ is the unique point at which $\Lambda_p^*$ attains its infimum on $\overline{D_{z}}$,}
\vspace{0.25cm}
\item[ii)]{$z^{**} = (z^q,1,1)$ is the unique point at which $\mathcal{I}_{\mathscr{S}}$ attains its infimum on $\overline{\mathscr{D}_{z}}$.}
\end{enumerate}
\end{lemma}
\begin{proof}
We start off by showing \textit{i)}. Let $t\in\R^2$ such that $t \in D_z^\circ$, meaning $t_1^{1/q}t_2^{-1/p}>z$. Then, for $\tilde z:= t_1^{1/q}t_2^{-1/p}$ we assume that $\tilde z \in \mathcal{J}_p$, as otherwise our claim trivially holds by Remark \ref{rmk:LegendreBijection}. We then have that $t \in \partial D_{\tilde z}$, thus, by Lemma \ref{lem:UniqueMinRF}, $\Lambda_p^*(t_1, t_2) > \Lambda_p^*(\tilde z ^q,1) = \mathcal{I}_{\|Z\|}(\tilde z)$. We know that $\Lambda_p^*$ is a convex function with a root in the expectation $(m_{p,q}^q,1)$ of the $V_i^{(n)}$ from \eqref{eq:Vn}, since it is a rate function (apply arguments from \cite[Lemma 2.2.5]{DZ} in $\R^2$). We also show in Lemma \ref{lem:Ableitungen} $ii)$ that $\mathcal{H}_x \Lambda_p^*(x) = \mathfrak{H}_{x}^{-1}$, and have argued in Section \ref{sec:Preliminaries} for why $\mathfrak{H}_{x}$ is positive definite on $\mathcal{D}_p$, hence the Hessian of $\Lambda_p^*$ is also positive definite on $\mathcal{J}_p$, giving us the strict convexity of $\Lambda_p^*$ and, thereby strict convexity of $\mathcal{I}_{\|Z\|}(z)  = \Lambda_p^*(z^*)$ on $\mathcal{J}_p$. Hence we know that $\mathcal{I}_{\|Z\|}(z)$ is strictly increasing in $z$ for $z >m_{p,q}$. Thus, as $\tilde z > z > m_{p,q}$, it follows that
\text{}\\
$$\Lambda_p^*(t_1, t_2) > \Lambda_p(\tilde z ^q,1) = \mathcal{I}_{\|Z\|}(\tilde z) > \mathcal{I}_{\|Z\|}(z) = \Lambda_p^*(z ^q,1)= \Lambda_p^*(z ^*),$$
\text{}\\
showing that $z^* = (z^q,1)$ minimizes $\Lambda_p^*$ over $\overline{D_z}$. The proof of \textit{ii)} is analogous, also using the strict monotonicity of the rate function.
\end{proof}
 Suppose that the distributions of  $S^{(n)}$ and $\mathscr{S}^{(n)}$ have respective densities $h^{(n)}$ and $\mathscr{h}^{(n)}$. Then we can formulate our probabilities of interest as %
\begin{equation} \label{eq:ProbRepDevInt}
\Pro\left(n^{1/p-1/q} \, \|Z^{(n)}\|_q > z \right) = \Pro\left(S^{(n)} \in D_{z}\right) =  \int_{D_{z}} h^{(n)}(x) \,  \dint x,
\end{equation}
and
\begin{equation} \label{eq:ProbRepDevIntUnp}
\Pro\left(n^{1/p-1/q} \, \|\mathscr{Z}^{(n)}\|_q > z\right) = \Pro\left(\mathscr{S}^{(n)} \in \mathscr{D}_{z}\right) =  \int_{\mathscr{D}_{z}} \mathscr{h}^{(n)}(x) \, \dint x.
\end{equation}
The following section will be devoted to showing the existence of these densities $h^{(n)}$ and $\mathscr{h}^{(n)}$ and presenting them explicitly, while Sections~\ref{sec:ProofMainResult1} and~\ref{sec:ProofMainResult2} will then approximate their integrals over their respective deviation areas $D_{z}$ and $\mathscr{D}_{z}$.%

\section{Joint Density Estimate}\label{sec:JointDensityEstimate}

Recalling the notation and definitions established in Section \ref{sec:Preliminaries}, we assume the same set-up as in Section \ref{sec:ProbRep} and can formulate the following local limit theorems for the densities $h^{(n)}$  and $\mathscr{h}^{(n)}$ of our probabilistic representations $S^{(n)}$ and $\mathscr{S}^{(n)}$. 
\begin{proposition}\label{prop:DensityEstimate}
 For $S^{(n)} = \frac{1}{n} \sum_{i=1}^n  V^{(n)}_i$ with $V^{(n)}_i = ({|Y^{(n)}_i|}^q, {|Y^{(n)}_i|}^p), \, Y_i^{(n)} \sim \bN_p$ i.i.d., and $x \in \mathcal{J}_p$, it holds that for sufficiently large $n \in \N$  the distribution of $S^{(n)} $ has Lebesgue density
	$$ h^{(n)}(x) = \frac{n}{2\pi} \, {(\det \mathfrak{H}_{x})}^{-1/2} \,  e^{- n \, \Lambda_p^*(x)} \, (1 + o(1)),$$
where $\mathfrak{H}_{x}:=  \mathcal{H}_{\tau}\Lambda_p(\tau(x))$ as in \eqref{eq:HessX}.
\end{proposition}
For the proof of this, we refer to the results of Borovkov and Rogozin \cite{BorovkovRogozin} or their convenient  reformulation in \cite[Theorem 3.1]{AdrianiBaldi}. Therein, a local density estimate is derived  for a sum of i.i.d. random vectors in $\R^d$ via the saddle point method. As discussed in Section \ref{subsec:PrelimLDP}, this means, one writes the density via the Fourier inversion theorem as a complex integral over its Fourier transform and then uses Cauchy's theorem to deform the path of integration, such that it passes through a complex saddle point. For sufficiently large $n \in \N$, the mass of the integral then heavily concentrates around that saddle point and standard integral expansion methods can be used to great effect. Naturally, this requires the conditions of the Fourier inversion theorem to be met, that is, the Fourier transform of the density has to be integrable. In \cite[Theorem 3.1]{AdrianiBaldi} this follows from the assumption that all the i.i.d. \hspace{-0.1cm}random vectors have a common bounded density, though it is noted in \cite[Remark 3.2]{AdrianiBaldi}, that this can be replaced by any argument ensuring that the Fourier inversion theorem can be applied. In our setting, the i.i.d. \hspace{-0.1cm}vectors are given by $V^{(n)}_i :=({|Y_i^{(n)}|}^q, {|Y_i^{(n)}|}^p)$, whose coordinates are highly dependent, thus such a density of the $V^{(n)}_i$ is not available. However, one can write the Fourier transform of $V^{(n)}_i$ with respect to the underlying distribution $\bN_p$ of the $Y_i^{(n)}$, and then infer integrability via the properties of its density $f_p$ and the Hausdorff-Young inequality, as was done by Liao and Ramanan in \cite[Lemma 6.1]{LiaoRamanan}. As the considered settings are quite similar, virtually the same arguments can be applied in our case, thereby making sure our referral to \cite[Theorem 3.1]{AdrianiBaldi} is indeed justified. 

\begin{proposition}\label{prop:DensityEstimateUnp}
 For $\mathscr{S}^{(n)} = \frac{1}{n} \sum_{i=1}^n  \mathscr{V}^{(n)}_i$ with $\mathscr{V}^{(n)}_i = ({|Y^{(n)}_i|}^q, {|Y^{(n)}_i|}^p, U^{1/n}), \, Y^{(n)}_i \sim \bN_p$ i.i.d., $U$ uniformly distributed on $[0,1]$ independently of the $Y_i^{(n)}$, and $x = (x_1, x_2) \in \mathcal{J}_p$, $y \in (0,1]$, it holds that for sufficiently large $n \in \N$  the distribution of $\mathscr{S}^{(n)}$ has Lebesgue density
	$$ \mathscr{h}^{(n)}(x_1, x_2, y) = \frac{n^2}{2\pi} \,  y^{-1}{(\det \mathfrak{H}_{x})}^{-1/2} \, e^{- n \, \mathcal{I}_{\mathscr{S}}(x_1, x_2, y)} \, (1 + o(1)),$$
where $\mathcal{I}_{\mathscr{S}}(x_1, x_2, y) := [\Lambda_p^*(x) - \log(y)]$ and $\mathfrak{H}_{x} :=  \mathcal{H}_{\tau}\Lambda_p(\tau(x))$ as in \eqref{eq:HessX}.
\end{proposition}
\begin{proof}
By direct calculation, we can see for $y\in[0,1]$ that $\Pro\left(U^{1/n} \le y\right) = \Pro\left(U \le y^n\right) = y^n$, giving that the density of $U^{1/n}$ is given by $f_{U^{1/n}}(y) = n \, y^{n-1}$. As $U^{1/n}$ is independent of the $Y_i^{(n)}$, and thereby also of $S^{(n)}=({|Y^{(n)}_i|}^q, {|Y^{(n)}_i|}^p)$, the density of $\mathscr{S}^{(n)} = \frac{1}{n} \sum_{i=1}^n  ({|Y^{(n)}_i|}^q, {|Y^{(n)}_i|}^p, U^{1/n})$ is given by the product of their densities, hence
$$ \mathscr{h}^{(n)}(x_1, x_2, y) = h^{(n)}(x_1, x_2) f_{U^{1/n}}(y) =\frac{n^2}{2\pi} \,  y^{-1}{(\det \mathfrak{H}_{x})}^{-1/2} \, e^{- n \, [\Lambda_p^*(x) - \log(y)]} \, (1 + o(1)).$$
This completes our proof.
\end{proof}
\vspace{0.25cm}
\section{Proof of the Main Result for $\ell_p^n$-spheres}\label{sec:ProofMainResult1}

In \eqref{eq:ProbRepDevInt} we have reformulated the deviation probability $\Pro\left(n^{1/p-1/q} \, \|Z^{(n)}\|_q > z\right) $ as an integral of the density estimate $h^{(n)}$ of the probabilistic representation $S^{(n)}$ over the deviation area  $D_{z}$. In Proposition \ref{prop:DensityEstimate} we have then given $h^{(n)}$ explicitly. For the proof of Theorem \ref{thm:MainResult1} it remains to calculate that integral. To do so, the integral will be split up into a neighbourhood $B_z$ of the point $z^*$, that has been shown in Lemma \ref{lem:UniqueMinCl} to be the infimum of $\Lambda_p^*$ over $\bar{D}_z$, and its complement $B_z^c$. The LDP from Proposition \ref{prop:KPT-LDP-Cnp} will be used to show the negligibility of the integral outside of the neighbourhood of $z^*$. Within the neighbourhood $B_z$, we use a result from Adriani and Baldi \cite{AdrianiBaldi}, which uses the Weingarten maps of the planar curves given by the boundary of $D_{z} \cap B_z$ and the level set of $\Lambda_p^*$ at $z^{*}$, to compute the integral. Following that, we will give these Weingarten maps explicitly, finishing our proof. %
\begin{proof}[Proof of Theorem \ref{thm:MainResult1}] 
We assume the set-up of Theorem \ref{thm:MainResult1} and use the reformulation \eqref{eq:ProbRepDevInt} to proceed by considering $ \Pro\left(S^{(n)} \in D_{z}\right)$. Let $B_z \subset \R^2$ be an open neighbourhood around $z^*$, small enough that $B_z \subset \mathcal{J}_p$. Then it holds that 
	\begin{equation} \label{eq:SplitIntegral}
	\Pro(S^{(n)} \in D_{z}) = \int_{D_{z}} h^{(n)}(x) \, \dint x =  \int_{D_{z} \cap \, B_z} h^{(n)}(x) \, \dint x +  \int_{D_{z} \cap \,  B_z^c} h^{(n)}(x) \, \dint x.
	\end{equation}
Since $z^* \notin B_z^c$, by Lemma \ref{lem:UniqueMinCl} , there exists an $\eta >0$, such that
$$ \inf_{y \in D_{z} \cap \, B_z^c} \Lambda_p^*(y) >\Lambda_p^*(z^*) + \eta,$$
and thus, by the LDP in Proposition \ref{prop:KPT-LDP-Cnp}, it holds that 
$$ \limsup_{n \to \infty} \frac{1}{n} \log \Pro(S^{(n)} \in D_{z} \cap \,B_z^c) \le - \inf_{y \, \in \,  D_{z} \cap \, B_z^c} \Lambda_p^*(y) \le - \Lambda_p^*(z^*) - \eta.$$
This gives us that 
\begin{equation} \label{eq:PropRepIntCompl}
\Pro \left(S^{(n)} \in D_{z} \cap \, B_z^c \right) \le e^{-n \, \Lambda_p^*(z^*) - n \, \eta} \,  (1+ o(1)) = \frac{1}{e^{n \,\eta}} \, e^{-n \Lambda_p^*(z^*)} (1+ o(1)) .
\end{equation} 
Furthermore, by our density estimate in Proposition \ref{prop:DensityEstimate}, it holds that 
	\begin{equation}\label{eq:AdrianiBaldiInt}
	\int_{D_{z} \cap \, B_z} h^{(n)}(x) \, \dint x = \frac{n}{2\pi} \, \int_{D_{z} \cap \, B_z}  {(\det \mathfrak{H}_{x})}^{-1/2} \, e^{- n \, \Lambda_p^*(x)}  \, \dint x  (1 + o(1)). 
	\end{equation}
	%\text{}\\
	To calculate this explicitly, we will rely on a technique established in \cite[Proof of Theorem 4.4]{AdrianiBaldi}. Therein, an asymptotic integral expansion of Bleistein and Handelsmann \cite[Equation (8.3.63)]{Bleistein} for Laplace  integrals is reformulated via the Weingarten maps of the integration area and the level set of the exponential function at its minimum, both seen as hypersurfaces. We will present it as one concise result, similar to that formulated in \cite[Lemma 4.6]{LiaoRamanan}.%\\
	%\\
	\begin{proposition} \label{prop:AdrianiBaldiWeingarten} Let $D\subset \R^d$ be a bounded domain such that $\partial D$ is a differentiable hypersurface in $\R^d$. Furthermore, let $g:\R^d \to \R$ be a differentiable function and $\phi: D \to [0, \infty)$ a nonnegative function that is twice differentiable and attains a unique infimum over $\overline{D}$ at $x^* \in \partial D$. Define the hypersurfaces %
	$$\mathscr{C}_{D}=\partial D  \qquad \text{ and } \qquad \mathscr{C}_\phi=\{x \in \R^d: \phi(x) = \phi(x^*)\},$$
	and denote by $L_{D}$ and $L_\phi$ their respective Weingarten maps at $x^*$. Then, for sufficiently large $n \in \N$, it holds that
	$$\int_{D} g(x) \, e^{-n\, \phi(x)} \, \dint x = \frac{{(2\pi)}^{(d-1)/2} \, \det(L_\phi^{-1}(L_\phi - L_{D}))^{-1/2}}{n^{(d+1)/2} \, \langle {\mathcal{H}_x\,\phi(x^*)}^{-1} \, \nabla_x \phi (x^*), \nabla_x \phi (x^*) \rangle^{1/2}} \, g(x^*) \, e^{-n \, \phi(x^*)}(1 + o(1)).$$
	\end{proposition}
	The proof of this is given by first applying the result from \cite[Equation (8.3.63)]{Bleistein} for Laplace-type integrals and then using the reformulation of the terms therein from \cite[Equation (4.6)]{AdrianiBaldi} with respect to the Weingarten map.\\
	\\
	Let us now check that the above conditions hold for the integral in \eqref{eq:AdrianiBaldiInt}. 	
	We have that $D_{z} \cap \, B_z$ is bounded, and for $z > m_{p,q}$, we can write $\partial D_{z}$ as the graph of the infinitely differentiable function $f:(0,\infty) \to (0,\infty)$ with $f(t_1)= z^{-p}t_1^{p/q}$, thus both $\partial D_{z}$ and $\partial (D_{z} \cap \, B_z)$ are differentiable planar curves. % 	
As discussed in Section \ref{sec:Preliminaries}, it holds for $x \in \mathcal{J}_p$ that $ g_x(\tau) := \langle x, \tau \rangle - \Lambda_p(\tau)$ has a unique supremum $\tau(x)$, i.e. $x - \nabla_\tau \Lambda_p(\tau)=0$ has a unique solution in $(x, \tau)$. It was also dicussed that  $\mathcal{H}_\tau \Lambda_p(\tau)$ is invertible for all $\tau \in \mathcal{D}_p$, thus, it follows from the implicit function theorem that $x \mapsto \tau(x)$ is as differentiable in $x$ as $(x, \tau) \mapsto (x - \nabla_\tau \Lambda_p(\tau))$ is in $\tau$. As $\Lambda_p$ is the logarithm of the joint m.g.f. of the $V_i^{(n)}$ from \eqref{eq:Vn}, the components of its derivatives are themselves infinitely differentiable within $\mathcal{D}_p$ by the standard properties of the m.g.f. (see e.g. \cite[Theorem 5.4]{Gupta2010}), yielding that $\tau(x)$ is infinitely differentiable on $\mathcal{J}_p$.  This overall gives us the infinite differentiability of $\Lambda_p(\tau(x))$ and thereby also of $\Lambda_p^*(x)$ on $\mathcal{J}_p$. For $B_z$ chosen small enough, it then follows for any $z > m_{p,q}$ with $z^* \in \mathcal{J}_p$ that $\Lambda_p^*$ is twice differentiable on $D_{z} \cap \, B_z$. %
 Nonnegativity of $\Lambda_p^*$ follows directly by the standard properties of rate functions (apply arguments from e.g. \cite[Lemma 2.2.5]{DZ} in $\R^2$). By the infinite differentiability of $\Lambda_p(\tau(x))$ in $x$, we get the differentiability of $g(x):={(\det \mathfrak{H}_{x})}^{-1/2} = {(\det\mathcal{H}_\tau\Lambda_p(\tau(x))}^{-1/2}$ in $x$. %
Lemma \ref{lem:UniqueMinCl} gives us the uniqueness of $z^*=(z^q,1) \,\in \partial (D_{z} \cap \, B_z)$ as an infimum on $\overline D_{z}$ and $\overline{D_{z} \cap \, B_z}$.
 
 Thus, in view of the above, we can use Proposition \ref{prop:AdrianiBaldiWeingarten} for $D = D_{z} \cap \, B_z \subset \R^2$ with $g(x)= {(\det \mathfrak{H}_{x})}^{-1/2}$, $\phi(x) = \Lambda_p^*(x)$, and $x^* = z^*$, and get that%
	\begin{eqnarray}\label{eq:IntPostWeingarten1}
	\displaystyle  \int_{D_{z} \cap \, B_z} h^{(n)}(x) \, \dint x &= \displaystyle \frac{n}{2\pi} \, \frac{{(2\pi)}^{1/2} \, \det(L_\Lambda^{-1}(L_\Lambda - L_{D}))^{-1/2} \, (\det \mathfrak{H}_{z^*})^{-1/2} \, e^{-n \, \Lambda_p^*(z^*)}}{n^{3/2} \, \langle {\mathcal{H}_x\,\Lambda_p^* (z^*)}^{-1} \, \nabla_x\,\Lambda_p^*(z^*), \nabla_x\,\Lambda_p^*(z^*) \rangle^{1/2}} \, (1 + o(1)), \quad
	\end{eqnarray}
	for the respective Weingarten maps at $z^*$ of the curves
	$$\mathscr{C}_{D}=\partial (D_{z} \cap \, B_z)  \qquad \text{ and } \qquad \mathscr{C}_{\Lambda}=\{x \in \R^2: \Lambda_p^*(x) = \Lambda_p^*(z^*)\}.$$
	Let us present the following identities for some of the terms in the fraction above, resulting from the definition of $\tau(x)$ and the properties of the Legendre-Fenchel transform:
	\begin{lemma} \label{lem:Ableitungen}
	It holds that
	
	\begin{enumerate}
	\item[i)]{$\nabla_{x} \Lambda_p^*(x) = \tau(x),$}\\
	\item[ii)]{$\mathcal{H}_{x}  \Lambda_p^*(x) = {\mathfrak{H}_{x}}^{-1}.$
	}
	\end{enumerate}
	\end{lemma}
	\begin{proof}
 We start by showing that $\nabla_{x} \, \Lambda_p^*(x) = \tau(x)$. We have defined $\tau(x)$ as  the supremum of $[\langle x, \tau \rangle - \Lambda_p(\tau)]$ in $\tau \in \R^2$ (see \eqref{eq:ArgSupLegendre}), thus it follows that 
 \begin{equation} \label{eq:Grad=0}
 \nabla_\tau \big[\langle x, \tau \rangle - \Lambda_p(\tau)\big]\big{|}_{\tau = \tau(x)} = x - \nabla_\tau \Lambda_p(\tau(x)) =0. 
 \end{equation}
With this, it follows that
 \begin{eqnarray*}
	 \nabla_{x}  \Lambda_p^*(x) &=& \nabla_{x} \big[ \langle x, \tau(x) \rangle - \Lambda_p(\tau(x))  \big] \vphantom{\sum}\\
	&=&  \tau(x) + J_{x} \tau(x) \, x  - \nabla_{x} \Lambda_p(\tau(x))\vphantom{\sum}\\
	&=&  \tau(x) + J_{x} \tau(x) \, x  -  J_{x} \tau(x) \nabla_{\tau} \Lambda_p(\tau(x))\vphantom{\sum} \\
	&=&  \tau(x)+ J_{x} \tau(x) \,   \big[x - \nabla_{\tau} \Lambda_p(\tau(x))\big] \vphantom{\sum}\\
	&=& \tau(x).  \vphantom{\sum}
\end{eqnarray*} 
 Let us now prove that $\mathcal{H}_{x} \Lambda_p^*(x) = {\mathfrak{H}_{x}}^{-1}$.  On the one hand, it follows from the above that  
\begin{equation}\label{eq:HesseMatId1}
\mathcal{H}_{x} \Lambda_p^*(x) = J_{x}\tau(x),
\end{equation}
while on the other hand, it holds that
\begin{eqnarray} \label{eq:HesseMatId2}
	\nonumber \mathcal{H}_{x} \Lambda_p^*(x) &=&  \mathcal{H}_{x} \big[ \langle x, \tau(x) \rangle - \Lambda_p(\tau(x))  \big]\vphantom{\sum}\\
	\nonumber&=&  \mathcal{H}_{x} \big[ \langle x, \tau(x) \rangle \big] - \mathcal{H}_{x} \big[\Lambda_p(\tau(x))  \big]\vphantom{\sum}\\
	\nonumber&=&  J_{x} \big[ \nabla_x \langle x, \tau(x) \rangle \big] - J_{x} \big[ \nabla_x \Lambda_p(\tau(x))  \big]\vphantom{\sum}\\
	\nonumber&=& J_{x}\big[\tau(x) + J_{x} \tau(x) x  \big]  - J_{x}\big[J_{x} \tau(x) \,  \nabla_{\tau} \Lambda_p(\tau(x))\big] \vphantom{\sum}\\
	\nonumber&=& J_{x}\tau(x) + J_{x} \big[J_{x} \tau(x) x  \big] - \mathcal{H}_{x}\tau(x) \, \nabla_{\tau} \Lambda_p(\tau(x)) - J_{x}\tau(x) J_{x}\big[\nabla_{\tau} \Lambda_p(\tau(x))\big]\vphantom{\sum}\\
	\nonumber&=& 2J_{x}\tau(x)  + \mathcal{H}_{x}\tau(x) \big[x - \nabla_{\tau} \Lambda_p(\tau(x))\big] - J_{x}\tau(x) \,J_{x}\tau(x) \mathcal{H}_\tau\Lambda_p(\tau(x)) \vphantom{\sum}\\
	&=& 2J_{x}\tau(x)   - J_{x}\tau(x) \, J_{x}\tau(x) \,   \mathcal{H}_{\tau} \Lambda_p(\tau(x)). \vphantom{\sum}
\end{eqnarray} 
Equating the terms \eqref{eq:HesseMatId1} and \eqref{eq:HesseMatId2} yields
\begin{center}
$
\begin{array}{crcl}
&J_{x}\tau(x) &=& 2J_{x}\tau(x)   - J_{x}\tau(x) \, J_{x}\tau(x) \,   \mathcal{H}_{\tau} \Lambda_p(\tau(x)) \vphantom{\sum \limits_0}\\
\Leftrightarrow& 0 &=& J_{x}\tau(x)   - J_{x}\tau(x) \, J_{x}\tau(x) \,   \mathcal{H}_{\tau} \Lambda_p(\tau(x)) \vphantom{\sum \limits_0}\\
\Leftrightarrow& 0 &=&  I_2   - J_{x}\tau(x) \,   \mathcal{H}_{\tau} \Lambda_p(\tau(x)) \vphantom{\sum \limits_0}\\
\Leftrightarrow& J_{x}\tau(x) &=&   \mathcal{H}_{\tau} \Lambda_p(\tau(x)) ^{-1},\\
\end{array}
$
\end{center}
where $I_2$ denotes the identity matrix in $\R^2$. Again using \eqref{eq:HesseMatId1} on the above yields
$$\mathcal{H}_{x} \Lambda_p^*(x) = J_{x}\tau(x) =  \mathcal{H}_{\tau} \Lambda_p(\tau(x)) ^{-1} =\mathfrak{H}_{x}^{-1},$$
and thereby finishes the proof. 
 \end{proof}
	Via Lemma \ref{lem:Ableitungen}, we get 
	$$ \Big\langle {\mathcal{H}_x\,\Lambda_p^*(z^*)}^{-1} \, \nabla_x\,\Lambda_p^*(z^*), \nabla_x\,\Lambda_p^*(z^*) \Big\rangle = \Big\langle \mathfrak{H}_{z^*} \, \tau(z^*), \tau(z^*) \Big\rangle.$$
	With the definition of $\displaystyle \xi(z)^2$ in \eqref{eq:NormTermXi} the integral in \eqref{eq:IntPostWeingarten1} hence simplifies as follows:
		\begin{eqnarray}\label{eq:IntPostWeingarten2}
	\displaystyle \int_{D_{z} \cap \, B_z} h^{(n)}(x) \, \dint x &= \displaystyle \frac{1}{\sqrt{2\pi n} \,  \xi(z)}  \, {\big(\det(L_\Lambda^{-1}(L_\Lambda - L_{D})\big)}^{-1/2} \, e^{-n \,  \Lambda_p^*(z^*)} \, (1 + o(1)).
	\end{eqnarray}
	We see that it only remains to prove that $\det(L_\Lambda^{-1}(L_\Lambda - L_{D})) = \kappa(z)^2$. We proceed to calculate the Weingarten maps of the curves $\mathscr{C}_{D}$ and $\mathscr{C}_{\Lambda}$ explicitly. As discussed in Section \ref{subsec:Weingarten}, the Weingarten map of a planar curve at a point $x$ reduces to the absolute value of its curvature in $x$.  As previously mentioned, $\partial D_{z}$ is the graph of a function $f:(0,\infty) \to (0,\infty)$ with $f(t_1)= z^{-p}t_1^{p/q}$. Thus, the same holds locally for $\mathscr{C}_{D}=\partial (D_{z} \cap \, B_z)$ in a neighbourhood of $z^*$, so by the curvature formula for graphs of functions, as seen in Remark \ref{rmk:Curvature} ii), it holds that
	$$
	L_D = \frac{|f^{\prime \prime} (z^q)|}{(1 + f^{\prime}(z^q)^2)^{3/2} },
	$$
	where
$$f^{\prime}(t_1)^2 = {\left(pq^{-1} \, z^{-p} \,  {t_1}^{(p/q) -1} \right)}^2 \Rightarrow f^{\prime}(z^q)^2 = p^2q^{-2} z^{-2q}, $$
and 
$$f^{\prime \prime}(t_1) = pq^{-1} \left(pq^{-1}-1\right) z^{-p} \,  {t_1}^{(p/q) -2} \Rightarrow f^{\prime \prime}(z^q) = (p^2 - pq)q^{-2} \, z^{-2q}.$$
This yields
\begin{equation}\label{eq:WeingartenD2}
L_D = \frac{|(p^2 - pq)q^{-2} \, z^{-2q}|}{(1 + p^2q^{-2} z^{-2q})^{3/2}} =  \frac{|pq(p-q)z^q|}{(z^{2q} + p^2q^{-2})^{3/2}}.
\end{equation}
The curve $\mathscr{C}_{\Lambda}$ is the zero set of the function $F(x) := \Lambda_p^*(x) - \Lambda_p^*(z^*)$. From Lemma \ref{lem:Ableitungen} we know that 
$$\displaystyle(F_{[1,0]}, F_{[0,1]}) = \Big(\frac{\partial}{\partial x_1} \Lambda_p^*(z^*), \, \frac{\partial}{\partial x_2} \Lambda_p^*(z^*) \Big) =  \tau(z^*)$$
and
$$\left(\begin{array}{rrr} 
F_{[2,0]} & F_{[1,1]} \\ 
F_{[1,1]} & F_{[0,2]}  \\ 
\end{array}\right)  =  \displaystyle \left(\begin{array}{rrr} 
\displaystyle \frac{\partial^2}{\partial^2 x_1} \Lambda_p^*(z^*)& \displaystyle \frac{\partial^2}{\partial x_2 \partial x_1 } \Lambda_p^*(z^*) \\ 
\\
\displaystyle \frac{\partial^2}{\partial x_1 \partial x_2} \Lambda_p^*(z^*) & \displaystyle \frac{\partial^2}{\partial^2 x_2} \Lambda_p^*(z^*) \\ 
\end{array}\right)  = \mathfrak{H}_{z^*}^{-1},$$
\text{}\\
for derivatives $F_{[i,j]}=F_{[i,j]}(z^*)$ as in \eqref{eq:AblNotation}. Hence, by the curvature formula for implicit curves from Lemma \ref{lem:ImplCurves} and Remark \ref{rmk:Curvature} i), we get 
\begin{equation}\label{eq:WeingartenLambda}
L_\Lambda =  \displaystyle \frac{\left |\tau(z^*)_2^2 \left(\mathfrak{H}_{z^*}^{-1}\right)_{11} - 2\tau(z^*)_1 \tau(z^*)_2 \left(\mathfrak{H}_{z^*}^{-1}\right)_{12} + \tau(z^*)_1^2 \left(\mathfrak{H}_{z^*}^{-1}\right)_{22}\right| }{{\big(\tau(z^*)_1^2 + \tau(z^*)_2^2\big)}^{3/2}}.
\end{equation}
Since both $L_D $ and $L_\Lambda$ are one-dimensional, it follows from \eqref{eq:WeingartenD2} and \eqref{eq:WeingartenLambda} that 
$$
\det(L_\Lambda^{-1}(L_\Lambda - L_{D})) = L_\Lambda^{-1}(L_\Lambda - L_{D}) = 1- \frac{L_D}{L_\Lambda} = \kappa(z)^2.
$$
for $\kappa(z)^2$ as in \eqref{eq:NormTermKappa}. It now follows with \eqref{eq:IntPostWeingarten2} that
\begin{eqnarray}\label{eq:PropRepIntWithin}
	\displaystyle \int_{D_{z} \cap \, B_z} h^{(n)}(x) \, \dint x &= \displaystyle \frac{1}{\sqrt{2\pi n} \, \xi(z) \, \kappa(z)}   e^{-n \Lambda_p^*(z^*)} \, (1 + o(1)).
	\end{eqnarray}
Comparing  \eqref{eq:PropRepIntWithin} with the upper bound of the integral outside of $B_z$ in \eqref{eq:PropRepIntCompl}, we can see that the integral over $B_z^c$ is negligible for large $n \in \N$. Thus, combining \eqref{eq:SplitIntegral}, \eqref{eq:PropRepIntCompl} and \eqref{eq:PropRepIntWithin} finishes the proof of Theorem \ref{thm:MainResult1}.%
\end{proof}

% ---------------------------------------------------------------------------------------------------------------------------

\section{Proof of the Main Result for $\ell_p^n$-balls}\label{sec:ProofMainResult2}
We use the notation and definitions established in Sections~\ref{sec:Preliminaries} through~\ref{sec:ProbRep}. Let $1 \le q < p < \infty $ and $z > m_{p,q}$ be such that $z^{*} \in \mathcal{J}_{p}$. We proceed similarly to the previous proof, using the reformulation of $\mathbb{P}\left (n^{1/p-1/q} \, \|\mathscr{Z}^{(n)}\|_{q} > z\right )$ from~\eqref{eq:ProbRepDevIntUnp} in conjunction with the density approximation from Proposition~\ref{prop:DensityEstimateUnp}. The resulting integral over $\mathscr{D}_{z}$ is again split into a neighbourhood of the minimum of $\mathcal{I}_{\mathscr{S}}$ over $\overline{\mathscr{D}}_{z}$ and its complement, which, according to Lemma~\ref{lem:UniqueMinCl}, is attained at $z^{**}=(z^{q},1,1)$. For the integral within that neighbourhood, we apply a result of Breitung and Hohenbichler \cite{Breitung1989}, which yields an integral approximation under less restrictive differentiability conditions than those in Proposition~\ref{prop:AdrianiBaldiWeingarten}. This result is again geometric in nature, as the behaviour of the density on $\partial \mathscr{D}_{z} $ still heavily dictates the value of the overall approximation. However, since this result is formulated for a certain neighbourhood of the origin, we first need to construct a sufficient transformation, mapping our deviation area into such a neighbourhood. After that, we calculate the specific approximation in our setting.
\begin{proof}[Proof of Theorem \ref{thm:MainResult2}] 
We assume the set-up of Theorem \ref{thm:MainResult2} and use the reformulation \eqref{eq:ProbRepDevIntUnp} to proceed by considering $ \Pro\left(\mathscr{S}^{(n)} \in \mathscr{D}_{z}\right)$. Let $\mathscr{B}_z \subset \R^3$  be an open neighbourhood around $z^{**}=(z^q,1,1)$ small enough that the first two coordinates of points within $\mathscr{B}_z$ lie in $\mathcal{J}_p$ and the third is positive. Then it holds by Proposition \ref{prop:DensityEstimateUnp} that
\begin{eqnarray} \label{eq:SplitIntegralUnp}
	\Pro\left(\mathscr{S}^{(n)} \in \mathscr{D}_z\right) &=& \int_{\mathscr{D}_{z} \cap \, \mathscr{B}_z } \mathscr{h}^{(n)}(x_1, x_2, y) \, \dint x_1 \dint x_2 \, \dint y +  \int_{\mathscr{D}_{z} \cap \,  \mathscr{B}_z^c} \mathscr{h}^{(n)}(x_1, x_2, y) \, \dint x_1 \dint x_2 \, \dint y . \quad
	\end{eqnarray}
As in the proof of Theorem \ref{thm:MainResult1}, we can follow from Lemma \ref{lem:UniqueMinCl} ii) and the LDP in Proposition \ref{prop:KPT-LDP-Unp} that there is an $\eta >0$, such that 
\begin{equation} \label{eq:PropRepIntComplUnp}
\Pro\left(\mathscr{S}^{(n)} \in \mathscr{D}_{z} \cap \, \mathscr{B}_z^c\right) \le e^{- n \, \mathcal{I}_{\mathscr{S}}(z^{**}) - n \eta} (1+ o(1)) = \frac{1}{e^{n \,\eta}} \, e^{-n \,  \Lambda_p^*(z^*)} (1+ o(1)) ,
\end{equation} 
with $\mathcal{I}_{\mathscr{S}}(t) = \big[ \Lambda_p^*(t_1, t_2) -  \log(t_3)\big],$ as defined in Lemma \ref{lem:UniqueMinRFUnp}. Let us now consider the first integral in \eqref{eq:SplitIntegralUnp}. Since $z^* \in \mathcal{J}_p$, for sufficiently small $\mathscr{B}_z$, we have that $x =(x_1, x_2) \in \mathcal{J}_p$ and $y \in (0,1]$. By the density approximation from Proposition \ref{prop:DensityEstimateUnp}, it then holds that 
	$$ \int_{\mathscr{D}_{z} \cap \, B_z} \hspace{-0.1cm} \mathscr{h}^{(n)}(x_1, x_2, y) \, \dint x_1 \, \dint x_2 \, \dint y =  \frac{n^2}{2\pi} \, \int_{\mathscr{D}_{z} \cap \, B_z}  \hspace{-0.1cm} y^{-1}{(\det \mathfrak{H}_{x})}^{-1/2} \, e^{- n \,  \mathcal{I}_{\mathscr{S}}(x_1, x_2, y)} \, \dint x_1 \, \dint x_2 \, \dint y \, (1 + o(1)).$$
As we have seen in Lemma \ref{lem:UniqueMinCl}, $\mathcal{I}_{\mathscr{S}}$ attains its infimum on $\overline{\mathscr{D}}_{z}$ at $z^{**}$.  However, we cannot use the result of Adriani and Baldi from Proposition \ref{prop:AdrianiBaldiWeingarten} here, since at $z^{**}$ the boundary of $\mathscr{D}_{z} \cap \mathscr{B}_z$ is not differentiable, and thereby not smooth. Hence, we use the following asymptotic integral approximation results based on Breitung and Hohenbichler \cite{Breitung1989}, which gives a Laplace integral approximation very similar to that in Liao and Ramanan \cite[Lemma 5.1]{LiaoRamanan}, but under weaker conditions.\\
\begin{proposition} \label{prop:Breitung}
		Let $F \subset \R^3$ be a bounded closed set containing the origin in its interior. If 
		\begin{enumerate}[label=(\alph*)]
		\item $f:F\to \R$ and $g:F \to \R$ are continuous functions with $g(\mathbf{0}) \ne 0$, where $\mathbf{0} := (0,0,0),$
		\item $f(x) > f(\mathbf{0})$ for all $x \in F \cap (\R^2_+ \times \R)\setminus\{0\}$,
		\item there is a neighbourhood $V \subset F$ of $\mathbf{0}$ in which $f$ is twice continuously differentiable,
		\item $f_{[1,0,0]}>0$, $f_{[0,1,0]} >0$, and $f_{[0,0,2]} >0$, with derivatives $f_{[i,j,k]} = f_{[i,j,k]}(\mathbf{0}) $ as in \eqref{eq:AblNotation},
		\end{enumerate}
		then it holds that 
		$$\int_{F \cap (\R^2_+ \times \R)} g(x) e^{-nf(x)} \, \dint x = \frac{\sqrt{2\pi}}{n^{5/2}} \, \frac{g(x^*)}{f_{[1,0,0]} f_{[0,1,0]} \sqrt{f_{[0,0,2]}}} e^{-n f(x^*)} (1+ o(1)), $$
\end{proposition}
\begin{rmk}
This is the result from \cite[Lemma 4]{Breitung1989} for $n=3, k=2$ and  functions $g$ and $(-f)$ instead of $h$ and $f$. The parameter $\lambda$ in our setting is replaced by the integer $n \in \N$. Furthermore, a typo within said result has been corrected, namely the sum in \cite[Equation (11)]{Breitung1989} is replaced by a product (compare proof therein). This proposition is quite close to \cite[Lemma 5.1]{LiaoRamanan}, but does not require the same level of smoothness of $f$ and $g$, and $g$ does not depend on $n\in \N$.
\end{rmk} 

To apply this, we use a transformation of $\mathscr{D}_{z} \cap \mathscr{B}_z$, mapping $z^{**}=(z^q, 1,1)$ to $\mathbf{0}$. Consider 
$$\mathfrak{I}: \R^3 \to \R^3 \qquad \text{ with } \qquad \mathfrak{I}(x_1, x_2, y) = (y^q x_1 - z^q x_2^{q/p}, 1-y, x_2 -1)= (t_1, t_2, t_3).$$
It then holds that $\mathfrak{I}(z^{**}) = \mathbf{0}$ and $\mathfrak{I}(\mathscr{D}_z) = \tilde{\mathscr{D}}_z := \{t \in \R^3: t_1 > 0, t_2 \in [0,1), t_3 > -1\}$.
 Furthermore,  in a neighbourhood of $z^{**}$ small enough such that $t_2 <1$, $\mathfrak{I}$ is invertible with 
$$\mathfrak{I}^{-1}(t_1, t_2, t_3) = \left(\frac{t_1 + z^q(t_3+1)^{q/p}}{(1-t_2)^{q}} , \,  t_3 + 1, \,  1-t_2\right)\hspace{-0.1cm}.$$
Let us calculate the Jacobian of $\mathfrak{I}^{-1}$: 
\begin{equation} \label{eq:JacobiTrsf}
J_{t} \mathfrak{I}^{-1}(t) = \left(\begin{array}{ccc} 
\frac{1}{(1-t_2)^q}& \frac{q(t_1 + z^q(t_3 + 1)^{q/p})}{(1-t_2)^{q+1}} &  \frac{  z^q \frac{q}{p}(t_3 + 1)^{(q/p) -1}}{(1-t_2)^q}\\ 
0 & 0  & 1\\ 
0 &-1 & 0 
\end{array}\right)\hspace{-0.1cm}. 
\end{equation}
Thus, we have that $ \displaystyle |\det J_{t} \mathfrak{I}^{-1}(t)| = (1-t_2)^{-q}$. We set $ \mathscr{g}(x_1, x_2, y) = y^{-1}{(\det \mathfrak{H}_{x})}^{-1/2}$, as well as $\tilde{\mathscr{B}}_z := \mathfrak{I}(\mathscr{B}_z)$, and transform the area of integration via $\mathfrak{I}^{-1}$, yielding%
\begin{eqnarray*}
\Pro \left(\mathscr{S}^{(n)} \in \mathscr{D}_z \cap \mathscr{B}_z\right) &=& \int_{\mathscr{D}_{z} \cap \, \mathscr{B}_z} \mathscr{h}^{(n)}(x_1, x_2, y) \, \dint x_1 \dint x_2 \, \dint y\\
\\
&=&  \frac{n^2}{2\pi} \, \int_{\mathscr{D}_{z} \cap \, \mathscr{B}_z}  y^{-1}{(\det \mathfrak{H}_{x})}^{-1/2} \, e^{- n \, [\Lambda_p^*(x_1, x_2) - \log(y)]} \, \dint x_1 \, \dint x_2 \, \dint y \, (1 + o(1))\\
\\
&=&  \frac{n^2}{2\pi} \, \int_{\mathscr{D}_{z} \cap \, \mathscr{B}_z}  \mathscr{g}(x_1, x_2, y) \, e^{- n \, \mathcal{I}_{\mathscr{S}}(x_1, x_2, y)} \, \dint x_1 \, \dint x_2 \, \dint y \, (1 + o(1))\\
\\
&=&\frac{n^2}{2\pi} \, \int_{\tilde{\mathscr{D}}_{z} \cap \, \tilde{\mathscr{B}}_z}  \mathscr{g}\circ\mathfrak{I}^{-1}(t) \, e^{- n \, \mathcal{I}_{\mathscr{S}}\circ\mathfrak{I}^{-1}(t) } \, (1-t_2)^{-q} \, \dint t \, (1 + o(1)).
\end{eqnarray*}
We now set $\tilde g(t) := (1-t_2)^{-q} \, \mathscr{g}\circ\mathfrak{I}^{-1}(t)$ and $\tilde f(t):= \mathcal{I}_{\mathscr{S}}\circ\mathfrak{I}^{-1}(t)$, then 
\begin{equation} \label{eq:PreLRInt}
\Pro\left(\mathscr{S}^{(n)} \in \mathscr{D}_z \cap \mathscr{B}_z\right) =\frac{n^2}{2\pi} \, \int_{\tilde{\mathscr{D}}_{z} \cap \, \tilde{\mathscr{B}}_z}  \tilde g(t) \, e^{- n \, \tilde f (t)} \, \dint t \, (1 + o(1)).
\end{equation}
We intend to apply Proposition \ref{prop:Breitung} to the integral in \eqref{eq:PreLRInt} for $F=\tilde{\mathscr{B}}_z$. It  holds that $\tilde{\mathscr{D}}_{z} \cap \, \tilde{\mathscr{B}}_z$ is bounded and since the value of the integral is the same if we integrate over the open set $\tilde{\mathscr{B}}_z$ or its closure, we will continue to work with $\tilde{\mathscr{B}}_z$. Further, we have that $\tilde{\mathscr{B}}_z$ contains the origin in its interior, as the interior point $z^{**}$ of ${\mathscr{B}}_z$ is again mapped by the continuous function $\mathfrak{I}$ onto an interior point, which is $\mathfrak{I}(z^{**}) = \mathbf{0}$. Since we have chosen the neighbourhood ${\mathscr{B}}_z$ of $z^{**}$ small enough for $\tilde{\mathscr{B}}_z$ to not contain $(t_1, 1, t_3)$, it holds that%
 $$\tilde g(t) = (1-t_2)^{-q} \, \left[(1- t_2)^{-1} \, {\left( \det \mathfrak{H}_{(\mathfrak{I}^{-1}(t)_1, \mathfrak{I}^{-1}(t)_3)}\right)}^{-1/2}\right]$$
 is also differentiable on $\tilde{\mathscr{D}_z} \cap \tilde{\mathscr{B}}_z$ as a composition of differentiable functions and thereby continuous on $\tilde{\mathscr{B}}_z$. The differentiability of $\mathfrak{I}^{-1}$, together with that of $\Lambda_p^*$ shown in the proof of Theorem \ref{thm:MainResult1}, yields the differentiability (and thereby the continuity) of $\tilde f(t):= \mathcal{I}_{\mathscr{S}}\circ\mathfrak{I}^{-1}(t)$ on  $\tilde{\mathscr{B}}_z$.  It holds furthermore that %
 \begin{eqnarray} \label{eq:GinNull}
\tilde g(\mathbf{0}) = {\left( \det \mathfrak{H}_{z^*}\right)}^{-1/2},
\end{eqnarray}
 which is positive, since $\mathfrak{H}_{z^*}$ is positive definite on $\mathcal{J}_p$, as discussed in Section \ref{sec:Preliminaries}. Again, for ${\mathscr{B}}_z$ small enough, it also holds (up to a null set) that $ \tilde{\mathscr{B}}_z \cap (\R^2_+ \times \R) = \tilde{\mathscr{B}}_z \cap \tilde{\mathscr{D}}_z$, on which we know from Lemma \ref{lem:UniqueMinRFUnp} and Lemma \ref{lem:UniqueMinCl} that $\mathbf{0} = \mathfrak{I}(z^{**}) $ is the unique infimum of $\tilde f$ since %
\begin{eqnarray} \label{eq:FinNull}
\tilde f(\mathbf{0}) = \mathcal{I}_{\mathscr{S}} \circ \mathfrak{I }^{-1}(\mathbf{0}) = \mathcal{I}_{\mathscr{S}}(z^{**}) = \Lambda_p^*(z^*).
\end{eqnarray}
 We can see from \eqref{eq:JacobiTrsf} that all partial derivatives of $\mathfrak{I}^{-1}$ are all themselves continuously differentiable in a sufficiently small neighbourhood of $\mathbf{0}$. Thereby, $\mathfrak{I}^{-1}$ is twice continuously differentiable in such a neighbourhood. The two-fold continuous differentiability of $\Lambda_p^*$ has already been shown in the proof of Theorem \ref{thm:MainResult1}. %
 Finally, by Lemma \ref{lem:Ableitungen} i), it holds that
\begin{eqnarray*} 
 \nabla_{(x_1, x_2, y)} \mathcal{I}_{\mathscr{S}}(z^{**}) &=& \left( \frac{\partial}{\partial x_1} \Lambda_p^*(x_1, x_2), \frac{\partial}{\partial x_2} \Lambda_p^*(x_1, x_2), -\frac{1}{y}\right)\Big{|}_{(x_1, x_2, y)=z^{**}}\\
 \\
 &=& \left( \tau(x)_1, \tau(x)_2, -\frac{1}{y}\right)\Big{|}_{(x_1, x_2, y)=z^{**}}\\
 \\
  &=& \left( \tau(z^*)_1, \tau(z^*)_2, -1\right),
\end{eqnarray*}
 from which we can deduce that %
\begin{eqnarray} \label{eq:AbleitungenF}
\nonumber \nabla_t \tilde{f}(\mathbf{0})  \,  &=&  \, \nabla_{(x_1, x_2, y)} \mathcal{I}_{\mathscr{S}}(z^{**}) \, J_{t} \mathfrak{I}^{-1}(\mathbf{0})\\
\nonumber\\
\nonumber&=& \left( \tau(z^*)_1, \tau(z^*)_2, -1\right) \,  \left(\begin{array}{ccc} 
1& q z^q & z^q  \, \frac{q}{p}\\ 
0 & 0  & 1\\ 
0 &-1 & 0 
\end{array}\right)\\
\nonumber\\
&=& \left(\tau(z^*)_1, \, \, qz^q \tau(z^*)_1+1,  \, \, z^q  \, \frac{q}{p} \tau(z^*)_1 + \tau(z^*)_2 \right).
\end{eqnarray}
It thereby follows that  $\nabla_t \tilde{f}(\mathbf{0}) \ne \mathbf{0}$, as the first two components cannot be equal to zero simultaneously. But since $\tilde{f}(t)$ attains its infimum on $ \tilde{\mathscr{B}}_z \cap (\R^2_+ \times \R)$ in $t = \mathbf{0}$, it holds that $\tilde f_{[1,0,0]}>0$ and $\tilde f_{[0,1,0]} >0$, as otherwise a step into either direction $t_1, t_2$ would maintain or decrease the value of $\tilde{f}$, contradicting the unique infimum property of $\mathbf{0}$. On the other hand, by the same argument, it has to hold that $\tilde f_{[0,0,1]} =0$ and $\tilde f_{[0,0,2]} >0$, as otherwise a step into either direction $t_3, (-t_3)$ would maintain or decrease $\tilde f$, again contradicting the unique infimum property of $\mathbf{0}$. Hence, we have shown all conditions for Proposition \ref{prop:Breitung}, whereby it now follows for the integral in \eqref{eq:PreLRInt} that
\begin{eqnarray} \label{eq:PostBreitungInt}
\nonumber \Pro\left(\mathscr{S}^{(n)} \in \mathscr{D}_z \cap \mathscr{B}_z\right) &=&\frac{n^2}{2\pi} \, \int_{\tilde{\mathscr{D}}_{z} \cap \, \tilde{\mathscr{B}}_z}  \tilde g(t) \, e^{- n \, \tilde f (t)} \, \dint t \, (1 + o(1))\\
\nonumber\\
\nonumber &=& \frac{n^2}{2\pi} \, \int_{\tilde{\mathscr{B}}_z \cap (\R^2_+ \times \R)}  \tilde g(t) \, e^{- n \, \tilde f (t)} \, \dint t \, (1 + o(1))\\
\nonumber\\
&=&\frac{1}{\sqrt{2\pi  n}} \, \frac{ \tilde g( \mathbf{0} )}{\tilde f_{[1,0,0]} \tilde f_{[0,1,0]} \sqrt{\tilde f_{[0,0,2]}}} e^{-n \tilde f(\mathbf{0})} (1+ o(1)).
\end{eqnarray}
The final term that remains to be calculated explicitly is $\tilde f_{[0,0,2]}$, as $\tilde f_{[1,0,0]}$ and  $\tilde f_{[0,1,0]}$ are given in \eqref{eq:AbleitungenF}.%
We start by noting that 
\begin{eqnarray*} 
\frac{\partial}{\partial t_3} \mathfrak{I}^{-1}(t) \Big{|}_{t = \mathbf{0}} =   \left(  \frac{  z^q \frac{q}{p}(t_3 + 1)^{(q/p) -1}}{(1-t_2)^q}, 1, 0\right)\Big{|}_{t = \mathbf{0}} =\left(z^q\frac{p}{q}, 1, 0\right),
\end{eqnarray*}
and
\begin{eqnarray*} 
\frac{\partial^2}{\partial^2 t_3} \mathfrak{I}^{-1}(t) \big{|}_{t = \mathbf{0}} = \left(  \frac{  z^q \frac{q}{p}\big(\frac{q}{p} -1\big)(t_3 + 1)^{(q/p) -2}}{(1-t_2)^q}, 0, 0\right)\Big{|}_{t = \mathbf{0}}= \left(\frac{z^q q^2}{p^2} -\frac{z^q q}{p}, 0, 0\right).
\end{eqnarray*}
By Lemma \ref{lem:Ableitungen} ii), we get that 
\vspace{-0.25cm}
$$\mathcal{H}_{(x_1, x_2, y)} \mathcal{I}_{\mathscr{S}}(z^{**}) = \left(\begin{array}{ccc} 
\vphantom{\int \limits^1} \left(\mathfrak{H}_{z^*}^{-1}\right)_{11}& \left(\mathfrak{H}_{z^*}^{-1}\right)_{12} & 0\\
\vphantom{\int \limits_0^1} \left(\mathfrak{H}_{z^*}^{-1}\right)_{21} & \left(\mathfrak{H}_{z^*}^{-1}\right)_{22}  & 0\\ 
0 & 0 & y^{-2}
\end{array}\right).$$
It thereby follows that
\begin{eqnarray} \label{eq:AblTilde3}
\nonumber \tilde f_{[0,0,2]}  \hspace{-0.2cm}&=& \hspace{-0.1cm} \frac{\partial^2}{\partial^2 t_3} \mathcal{I}_{\mathscr{S}}\circ\mathfrak{I}^{-1}(\mathbf{0})\\
\nonumber \\
\nonumber &=&  \hspace{-0.1cm} \frac{\partial}{\partial t_3}\left[ \nabla_{(x_1, x_2, y)} \mathcal{I}_{\mathscr{S}}(\mathfrak{I}^{-1}(t)) \, \frac{\partial}{\partial t_3} \mathfrak{I}^{-1}(t)\right] \Bigg{|}_{t = \mathbf{0}}\\
\nonumber \\
\nonumber &=&  \hspace{-0.1cm} \frac{\partial}{\partial t_3}\left[ \nabla_{(x_1, x_2, y)} \mathcal{I}_{\mathscr{S}}(\mathfrak{I}^{-1}(t))\right]\Big{|}_{t = \mathbf{0}}   \, \, \frac{\partial}{\partial t_3} \mathfrak{I}^{-1}(\mathbf{0}) +   \nabla_{(x_1, x_2, y)} \mathcal{I}_{\mathscr{S}}(z^{**})  \, \,  \frac{\partial^2}{\partial^2 t_3} \mathfrak{I}^{-1}(\mathbf{0})\\
\nonumber \\
\nonumber &=&  \hspace{-0.1cm} \left(z^q\frac{q}{p}, 1, 0\right)  \mathcal{H}_{(x_1, x_2, y)} \mathcal{I}_{\mathscr{S}}(z^{**}) \,  \,    \left(z^q\frac{q}{p}, 1, 0\right)  + \Big( \tau(z^*)_1, \tau(z^*)_2, -1\Big)   \left(\frac{z^q q^2}{p^2} -\frac{z^q q}{p}, 0, 0\right)\\
\nonumber \\
\nonumber &=& \hspace{-0.1cm} \left(\frac{z^q q}{p}, 1, 0\right)    \left(\frac{z^q q}{p}\left(\mathfrak{H}_{z^*}^{-1}\right)_{11} + \left(\mathfrak{H}_{z^*}^{-1}\right)_{12},  \frac{ z^q q}{p} \left(\mathfrak{H}_{z^*}^{-1}\right)_{21} + \left(\mathfrak{H}_{z^*}^{-1}\right)_{22},  0\right) \hspace{-0.05cm}+  \tau(z^*)_1 \Big(\frac{z^q q^2}{p^2} -\frac{z^q q}{p}\Big) \\
\nonumber \\
&=& \hspace{-0.1cm} \frac{z^{2q} q^2}{p^2}\left(\mathfrak{H}_{z^*}^{-1}\right)_{11} + \frac{2z^q q}{p} \left(\mathfrak{H}_{z^*}^{-1}\right)_{12} +  \left(\mathfrak{H}_{z^*}^{-1}\right)_{22} + \, \tau(z^*)_1 \Big(\frac{z^q q^2}{p^2} -\frac{z^q q}{p}\Big).   \vphantom{\int\limits_0}
\end{eqnarray}
Plugging the terms from \eqref{eq:GinNull}, \eqref{eq:AbleitungenF} and \eqref{eq:AblTilde3} into the fraction in \eqref{eq:PostBreitungInt}, we get that 
\begin{eqnarray} \label{eq:FracLR}
\nonumber\frac{\tilde g(\mathbf{0})}{\tilde f_{[1,0,0]} \tilde f_{[0,1,0]} \sqrt{|\tilde f_{[0,0,2]}|}} &=& {\left( \det \mathfrak{H}_{z^*}\right)}^{-1/2} \,  (\tau(z^*)_1)^{-1} \, (qz^q \tau(z^*)_1+1)^{-1}\\
\nonumber&& \times \, {\Bigg[   \frac{z^{2q} q^2}{p^2}\left(\mathfrak{H}_{z^*}^{-1}\right)_{11} + \frac{2z^q q}{p} \left(\mathfrak{H}_{z^*}^{-1}\right)_{12} +  \left(\mathfrak{H}_{z^*}^{-1}\right)_{22} + \, \tau(z^*)_1 \Big(\frac{z^q q^2}{p^2} -\frac{z^q q}{p} \Big)\Bigg]}^{-1/2}  \vphantom{\int\limits_0}\\
\nonumber&=& \Bigg[ \det \mathfrak{H}_{z^*} \,  \left(\tau(z^*)_1\right)^{2} \, \left(qz^q \tau(z^*)_1+1 \right)^{2}\\
\nonumber&& \times \, \Bigg(   \frac{z^{2q} q^2}{p^2}\left(\mathfrak{H}_{z^*}^{-1}\right)_{11} + \frac{2z^q q}{p} \left(\mathfrak{H}_{z^*}^{-1}\right)_{12} +  \left(\mathfrak{H}_{z^*}^{-1}\right)_{22} + \, \tau(z^*)_1 \frac{z^q q(q-p)}{p^2} \Bigg{)} \Bigg]^{-1/2}   \vphantom{\int\limits_0}\\
&=& \gamma(z)^{-1},
\end{eqnarray}
with $\gamma(z)$ as in \eqref{eq:NormTermGamma}. Hence, it follows with \eqref{eq:FinNull}, \eqref{eq:PostBreitungInt}, and \eqref{eq:FracLR} that 
\begin{equation} \label{eq:UnpIntInBz}
 \Pro\left(\mathscr{S}^{(n)} \in \mathscr{D}_z \cap \mathscr{B}_z\right) = \frac{1}{\sqrt{2\pi n} \, \gamma(z)} e^{-n \,  \Lambda_p^*(z^*)} \ (1+ o(1)).
\end{equation}
Combining the representation from \eqref{eq:SplitIntegralUnp} with the two integral estimates from \eqref{eq:PropRepIntComplUnp} and \eqref{eq:UnpIntInBz} shows that the integral in the complement of $\mathscr{B}_z$ can be neglected and we have that 
$$
	\Pro\left(n^{1/p-1/q} \|\mathscr{Z}^{(n)}\|_q > z\right) = \Pro\left(\mathscr{S}^{(n)} \in \mathscr{D}_z \cap \mathscr{B}_z\right) =  \frac{1}{\sqrt{2\pi n} \, \gamma(z)} \, e^{-n \, \Lambda_p^*(z^*)} \, (1 + o(1)),
$$
which proves our second main result for $\ell_p^n$-balls.
\end{proof}
\text{}
\section*{Appendix}\label{sec:Appendix}
\begin{proof}[Proof of Lemma \ref{lem:UniqueMinRF}]
Let $z >m_{p,q}$ such that $z^*= (z^q, 1) \in \mathcal{J}_p$. Then it holds that 
$$\displaystyle \mathcal{I}_{\|Z\|}(z) = \inf_{{\text{$t_1, t_2 >0$}}\atop{\text{$t_1^{1/q}t_2^{-1/p} = z$}}}  \Lambda_p^*(t_1, t_2) = \inf_{\tilde{t}_1,\tilde{t}_2 >0:  \\ \, \tilde{t}_1=z \, \tilde{t}_2} \Lambda_p^*(\tilde{t}_1^q, \tilde{t}_2^p) = \inf_{\tilde{t}_2 >0} \Lambda_p^*(z^q {\tilde{t}_2}^{q}, {\tilde{t}_2}^p).$$
We set $t_z := (z^q \tilde{t}_2^{q}, \tilde{t}_2^p)$, then with \eqref{eq:ArgSupLegendre} it follows that 
$$\displaystyle \mathcal{I}_{\|Z\|}(z) = \inf_{\tilde{t}_2 >0} \sup_{s \in \R^2} \left( \langle s, t_z \rangle - \Lambda_p(s) \right) = \inf_{\tilde{t}_2 >0}   \Big[\langle \tau(t_z), t_z \rangle - \Lambda_p(\tau(t_z)) \Big].$$
	Our goal is to show that the infimum is attained at $ t_z^* := z^*$, i.e. at $\tilde{t}_2=1$. Recall the definition $g_t(s) := \langle s, t \rangle - \Lambda_p(s)$ for $t \in \mathcal{J}_p$ from Section \ref{subsec:LpLDPs}. By the definition of $\tau(t_z)$ it holds that $g_{t_z}(s)$ attains its supremum at $\tau(t_z)$, thus it holds that $\nabla_s  \, g_{t_z}(s) \big{|}_{s=\tau(t_z)}  = t_z - \nabla_s \Lambda_p(s)\big{|}_{s=\tau(t_z)} =0,$ which gives
\vspace{-0.25cm}
	\begin{eqnarray}\label{eq:GradGLS}
 t_z = \left ( z^q \tilde{t}_2^q, \tilde{t}_2^p \right) =  \left(\frac{\partial}{\partial s_1} \Lambda_p(s)\big{|}_{s=\tau(t_z)},  \frac{\partial}{\partial s_2} \Lambda_p(s)\big{|}_{s=\tau(t_z)}  \right).
	\end{eqnarray} 
	\text{}\vspace{-0.1cm}\\
We now aim to write $\frac{\partial}{\partial s_2} \Lambda_p(s)$ with respect to $\frac{\partial}{\partial s_1} \Lambda_p(s)$ and then use the above equations. To do so, we firstly want to reformulate $\Lambda_p$  along the lines of \cite[Lemma 5.7]{GKR}. It holds that 
$$\displaystyle \Lambda_p(s) := \log \int_\R e^{s_1 |y|^q + s_2 |y|^p} f_p(y) \, \dint y = \log \left( \frac{1}{2p^{1/p}\Gamma\big(1+\frac{1}{p}\big)}\, \int_\R e^{s_1 |y|^q - \frac{1}{p}(1 - ps_2) |y|^p} \, \dint y\right)\hspace{-0.1cm}.$$
The change of variables $x = (1-ps_2)^{1/p} y$ then gives 
$$\hspace{0.1cm} \displaystyle \Lambda_p(s) = \log \left( (1-ps_2)^{-1/p}  \int_\R e^{\frac{s_1}{(1-ps_2)^{q/p}} |x|^q } f_p(x) \, \dint x \right) = -\frac{1}{p} \log (1-ps_2) + \log \varphi_{|X|^q}\left(\frac{s_1}{(1-ps_2)^{q/p}}\right) \hspace{-0.1cm},$$
where $\varphi_{|X|^q}$ is the m.g.f. \hspace{-0.05cm}of a random variable $|X|^q$ with $X \sim \mathbf{N}_p$. Hence, 
\begin{eqnarray*}
\nonumber \frac{\partial}{\partial s_1} \Lambda_p(s) &=& \frac{\partial}{\partial s_1}  \left[ \log \varphi_{|X|^q}\left(\frac{s_1}{(1-ps_2)^{q/p}}\right) \right]\\
\nonumber \\
\nonumber &=&\varphi_{|X|^q}\left(\frac{s_1}{(1-ps_2)^{q/p}}\right)^{-1} \frac{\partial}{\partial s_1}  \left[ \varphi_{|X|^q}\left(\frac{s_1}{(1-ps_2)^{q/p}}\right) \right]\\
\nonumber \\
\nonumber &=&\varphi_{|X|^q}\left(\frac{s_1}{(1-ps_2)^{q/p}}\right)^{-1}  \int_\R (1-ps_2)^{-q/p} \,  |x|^q \, e^{\frac{s_1}{(1-ps_2)^{q/p}} |x|^q} f_p(x) \, \dint x \\
\nonumber \\
 &=& (1-ps_2)^{-q/p} \varphi_{|X|^q}\left(\frac{s_1}{(1-ps_2)^{q/p}}\right)^{-1}  \varphi_{|X|^q}^{\prime} \left(\frac{s_1}{(1-ps_2)^{q/p}}\right),
\end{eqnarray*} 
where $\varphi_{|X|^q}^{\prime} \left(\frac{s_1}{(1-ps_2)^{q/p}}\right) = \varphi_{|X|^q}^{\prime} \left(t\right) \Big|_{t = \frac{s_1}{(1-ps_2)^{q/p}}}$. Moreover, with the above we get that 
\begin{eqnarray}\label{eq:DerivativeRewrite2}
	\nonumber \frac{\partial}{\partial s_2} \Lambda_p(s) &=& {(1-ps_2)}^{-1}  + \frac{\partial}{\partial s_2}  \left[ \log \varphi_{|X|^q}\left(\frac{s_1}{(1-ps_2)^{q/p}}\right) \right]\\
\nonumber	\\
\nonumber	&=& {(1-ps_2)}^{-1}  +  \varphi_{|X|^q}\left(\frac{s_1}{(1-ps_2)^{q/p}}\right)^{-1} \frac{\partial}{\partial s_2} \left[  \varphi_{|X|^q}\left(\frac{s_1}{(1-ps_2)^{q/p}}\right) \right]\\
\nonumber	\\
\nonumber	&=&{(1-ps_2)}^{-1}  + \varphi_{|X|^q}\left(\frac{s_1}{(1-ps_2)^{q/p}}\right)^{-1}  \int_\R \frac{q s_1}{(1-ps_2)^{(q+p)/p}} \,  |x|^q \, e^{\frac{s_1}{(1-ps_2)^{q/p}} |x|^q} f_p(x) \, \dint x \\
\nonumber	\\
\nonumber	&=& {(1-ps_2)}^{-1}  +  \frac{q s_1}{(1-ps_2)^{(q+p)/p}} \,\varphi_{|X|^q}\left(\frac{s_1}{(1-ps_2)^{q/p}}\right)^{-1}  \varphi_{|X|^q}^{\prime} \left(\frac{s_1}{(1-ps_2)^{q/p}}\right)\\
\nonumber	\\
\nonumber	&=& {(1-ps_2)}^{-1}  +  \frac{q s_1}{1-ps_2} \, (1-ps_2)^{-q/p} \,\varphi_{|X|^q}\left(\frac{s_1}{(1-ps_2)^{q/p}}\right)^{-1}  \varphi_{|X|^q}^{\prime} \left(\frac{s_1}{(1-ps_2)^{q/p}}\right)\\
\nonumber	\\
&=& {(1-ps_2)}^{-1}  +  q s_1{(1-ps_2)}^{-1} \, \frac{\partial}{\partial s_1} \Lambda_p(s).
\end{eqnarray} 
Plugging in the identities from \eqref{eq:GradGLS} into \eqref{eq:DerivativeRewrite2} it follows for $(s_1, s_2) = \big(\tau(t_z)_1, \tau(t_z)_2\big)$:
\begin{equation} \label{eq:DerivativeRewrite3}
 \displaystyle   \tilde{t}_2^p =  \displaystyle {(1-p\tau(t_z)_2)}^{-1}  +  q \tau(t_z)_1 {(1-p\tau(t_z)_2)}^{-1} \, z^q \tilde{t}_2^q. 
\end{equation} 
Using this, we can calculate the derivative of $\Lambda_p^*(t_z)$ in $t$ (we write $t$ instead of $\tilde{t}_2$ for notational brevity), where $\tau(t_z)$ is considered as a function in $t$ as well. It holds that 
\begin{eqnarray}\label{eq:AblLegendre}
\nonumber \frac{\partial}{\partial t} \Lambda_p^*(t_z) \hspace{-0.1cm} &=& \hspace{-0.1cm}\frac{\partial}{\partial t} \Lambda_p^*(z^q t^q, t^p) \vphantom{\sum \limits_0}\\
%\nonumber \\
\nonumber &=& \hspace{-0.1cm}\frac{\partial}{\partial t} \big[ \langle t_z, \tau(t_z) \rangle - \Lambda_p(\tau(t_z)) \big] \vphantom{\sum \limits_0}\\
%\nonumber \\
\nonumber &=& \hspace{-0.1cm}\frac{\partial}{\partial t} \big[ z^q t^q \tau(t_z)_1 + t^p \tau(t_z)_2 - \Lambda_p(\tau(t_z)) \big] \vphantom{\sum\limits_0}\\
%\nonumber \\
\nonumber &=&   \hspace{-0.1cm}z^q q t^{q-1} \tau(t_z)_1 + z^q t^q \frac{\partial}{\partial t} \tau(t_z)_1 + p t^{p-1} \tau(t_z)_2 + t^p \frac{\partial}{\partial t} \tau(t_z)_2 -  \frac{\partial}{\partial t} \Lambda_p(\tau(t_z)) \vphantom{\sum\limits_0}\\
%\nonumber\\
\nonumber &=&   \hspace{-0.1cm}z^q q t^{q-1} \tau(t_z)_1 + z^q t^q \frac{\partial}{\partial t} \tau(t_z)_1 + p t^{p-1} \tau(t_z)_2 + t^p \frac{\partial}{\partial t} \tau(t_z)_2 -  J_t(\tau(t_z)) \nabla_s \Lambda_p(s)\big{|}_{s=\tau(t_z)} \vphantom{\sum\limits_0}\\
%\nonumber\\
\nonumber &=&   \hspace{-0.1cm}z^q q t^{q-1} \tau(t_z)_1 + z^q t^q \frac{\partial}{\partial t} \tau(t_z)_1 + p t^{p-1} \tau(t_z)_2 + t^p \frac{\partial}{\partial t} \tau(t_z)_2 \vphantom{\sum\limits_0}\\
 %\nonumber \\
\nonumber && \hspace{-0.1cm} - \frac{\partial}{\partial t} \tau(t_z)_1 \frac{\partial}{\partial s_1} \Lambda_p(s)\big{|}_{s=\tau(t_z)} - \frac{\partial}{\partial t} \tau(t_z)_2 \frac{\partial}{\partial s_2} \Lambda_p(s)\big{|}_{s=\tau(t_z)}.
\end{eqnarray} 
\text{}\\
We now use the identity from \eqref{eq:GradGLS}, which yields
\begin{eqnarray}\label{eq:AblLegendre2}
\nonumber \frac{\partial}{\partial t} \Lambda_p^*(t_z)  \hspace{-0.1cm}&=& \hspace{-0.1cm}  z^q q t^{q-1} \tau(t_z)_1 + z^q t^q \frac{\partial}{\partial t} \tau(t_z)_1 + p t^{p-1} \tau(t_z)_2 + t^p \frac{\partial}{\partial t} \tau(t_z)_2 - \frac{\partial}{\partial t} \tau(t_z)_1  z^q t^q - \frac{\partial}{\partial t} \tau(t_z)_2 t^p\vphantom{\sum\limits_0} \\
%\nonumber \\
&=& \hspace{-0.1cm}z^q q t^{q-1} \tau(t_z)_1 + p t^{p-1} \tau(t_z)_2.
\end{eqnarray} 
\text{}\\
Reformulating the identity in \eqref{eq:DerivativeRewrite3} yields
\begin{equation}\label{eq:DerivativeRewrite4}
 \displaystyle   t^p =  \displaystyle {(1-p\tau(t_z)_2)}^{-1}  +  q \tau(t_z)_1{(1-p\tau(t_z)_2)}^{-1} \, z^q t^q \Leftrightarrow  (1-p\tau(t_z)_2) t^{p-1} - t^{-1}= z^q t^{q-1}q \tau(t_z)_1.  
\end{equation} 
%\text{}\\
Thus, if we set $\frac{\partial}{\partial t} \Lambda_p^*(t_a) =0$, we get from \eqref{eq:AblLegendre2} and \eqref{eq:DerivativeRewrite4} that  
\begin{eqnarray*}\label{eq:AblLegendre3}
\nonumber \frac{\partial}{\partial t} \Lambda_p^*(t_z)  =0 & \Leftrightarrow & 0= z^q q t^{q-1} \tau(t_z)_1 + p t^{p-1} \tau(t_z)_2 \vphantom{\sum}\\
%\\
& \Leftrightarrow & 0=  (1-p\tau(t_z)_2) t^{p-1} - t^{-1} + p t^{p-1} \tau(t_z)_2 \vphantom{\sum\limits_0}\\ 
%\\
& \Leftrightarrow & t=  1.
\end{eqnarray*} 
Hence, the infimum of $\Lambda_p^*$ over $\partial D_{z}$ is attained at $t_z^* = (z^q, 1) = z^*$ Since $\Lambda_p^*$ is strictly convex (see properties of the Legendre-Fenchel transform), this minimum is unique. Thereby, our claim is proven.
 \end{proof}
%
%-----------------------------------------------------------------------------------------------------------
%
%
\begin{proof}[Proof of Lemma \ref{lem:UniqueMinRFUnp}]
Let $ z > m_{p,q}$ such that $z^{*} = (z^q, 1) \in \mathcal{J}_p$. Furthermore, set $z^{**} := (z^q,1,1)$ and $\mathcal{I}_{\mathscr{S}}(t) := [ \Lambda_p^*(t_1, t_2) -  \log(t_3)]$, $t\in \R^3$. We use the definitions of $\mathcal{I}_{\|Z\|}$ and $\mathcal{I}_{U}$, together with Lemma  \ref{lem:UniqueMinRF}, to get that 
\begin{eqnarray*}\mathcal{I}_{\|\mathscr{Z}\|}(z) &=& \displaystyle \inf_{{\text{$z = t_1^{1/q} t_2^{-1/p}t_3$}}\atop{\text{$t_1, t_2 >0, t_3 \in (0,1]$}}} \mathcal{I}_{\mathscr{S}}(t)\\
\\
&=& \displaystyle \inf_{{\text{$z = z_1 z_2$}}\atop{\text{$z_1>0, z_2 \in (0,1]$}}} \left[  \inf_{{\text{$t_1, t_2 >0$}}\atop{\text{$t_1^{1/q}t_2^{-1/p} = z_1$}}} \Lambda_p^*(t_1,t_2) + \mathcal{I}_U(z_2)\right] \\
\\
&=&  \displaystyle \inf_{{\text{$z = z_1 z_2$}}\atop{\text{$z_1 >0, z_2 \in (0,1]$}}} \left[ \Lambda_p^{*}(z_1^q,1) - \log(z_2)\right].
\end{eqnarray*}
By the same arguments as in the proof of Lemma \ref{lem:UniqueMinCl}, we know that $\mathcal{I}_{\|Z\|}(z)  = \Lambda_p^*(z^q, 1)$ is strictly convex in $z$ on $\mathcal{J}_p$ with a unique root in $m_{p,q}$. Hence, it follows that for $z> m_{p,q}$ with $z \in \mathcal{J}_p$ it holds that $\mathcal{I}_{\|Z\|}(z) = \Lambda_p^*(z^q,1)$ is strictly increasing in $z$. Since $z_2 \le 1$, $z=z_1 z_2$, and $1 < q$, we have $z_1^q \ge z > m_{p,q}$, meaning that $\Lambda_p^*(z_1^q,1)$ is strictly increasing in $z_1$. Furthermore, we can see that $- \log(z_2)$ is strictly decreasing in $z_2$. Hence, rewriting $z_1$ with respect to $z_2$ then gives\\
\begin{eqnarray*}
\mathcal{I}_{\|\mathscr{Z}\|}(z) &=&   \inf_{{\text{$z_1 = z/ z_2$}}\atop{\text{$z_2 \in (0,1]$}}} \left[ \Lambda_p^{*}\Big(\Big(\frac{z}{z_2}\Big)^q,1\Big) - \log(z_2)\right],
 \end{eqnarray*}
which is strictly decreasing in $z_2$. Thus, choosing $z_2=1$ gives $z_1 = z$ and 
$$\mathcal{I}_{\|\mathscr{Z}\|}(z) =  \mathcal{I}_{\mathscr{S}}(z^{**}) = \Lambda_p^*(z^*),$$
finishing the proof.
\end{proof}
%
%--------------------------------------------------------
%
 
\section*{Acknowledgments}\label{sec:Ack}

The author would like to thank Kavita Ramanan and Joscha Prochno for the insightful exchanges on the topic of sharp large deviations in asymptotic geometric analysis. Furthermore, the author would like to thank his supervisor Christoph Thäle for the helpful discussions, feedback and constructive criticism throughout the writing of this paper.

\end{document}